\documentclass[12pt,a4paper]{article}

\usepackage[utf8]{inputenc}
\usepackage[english]{babel}
\usepackage{mathtools} 
\usepackage{amsfonts} 
\usepackage{amsmath}
\usepackage{amssymb}
\usepackage{amsthm}
\usepackage{mathrsfs}
\usepackage{microtype}
\usepackage{marvosym}
\usepackage{wasysym} 
\usepackage{stmaryrd}
\usepackage{dsfont}
\usepackage{enumerate}
\usepackage{multicol,multirow}
\usepackage{graphicx}
\usepackage{xcolor}
\usepackage{hyperref}
\hypersetup{
	colorlinks=false,
	linkcolor=blue,
	urlcolor=red,
	linktoc=all}

\newcommand{\du}{\,\textrm{d}u}

\newcommand{\Hilb}{\mathcal{H}}

\newcommand{\eps}{\text{\Large{$\varepsilon$}}}
\newcommand{\oneop}{\mathds{1}}

\newcommand{\rhobar}{\overline\rho}
\newcommand{\sigmabar}{\overline\sigma}

\newcommand{\Mob}{\mathsf{M\ddot ob}}

\newcommand{\Sone}[1][]{\mathbb{S}^{1#1}}

\newcommand{\C}{\mathcal{C}}
\newcommand{\Z}{\mathcal{Z}}
\renewcommand{\O}{\mathcal{O}}
\newcommand{\cF}{\mathcal{F}}

\newcommand{\cU}{\mathcal{U}}

\newcommand{\K}{\mathcal{K}}

\newcommand{\D}{\mathcal{D}}
\newcommand{\cP}{\mathcal{P}}

\newcommand{\cS}{\mathcal{S}}

\newcommand{\A}{\mathcal{A}}
\newcommand{\B}{\mathcal{B}}
\newcommand{\Aloc}{\A_{\loc}}
\newcommand{\Aholo}{\A_{\holo}}
\newcommand{\cI}{\mathcal{I}}

\newcommand{\M}{\mathcal{M}}

\newcommand{\N}{\mathcal{N}}
\newcommand{\cR}{\mathcal{R}}

\newcommand{\RR}{\mathbb{R}}

\newcommand{\CC}{\mathbb{C}}

\newcommand{\NN}{\mathbb{N}}

\DeclareMathOperator{\End}{End}
\DeclareMathOperator{\Ind}{Ind}
\DeclareMathOperator{\loc}{loc}
\DeclareMathOperator{\holo}{holo}

\DeclareMathOperator{\Hom}{Hom}

\DeclareMathOperator{\DHR}{DHR}

\DeclareMathOperator{\catVec}{Vec}

\DeclareMathOperator{\vN}{vN}
\DeclareMathOperator{\Ad}{Ad}

\DeclareMathOperator{\id}{id}
\DeclareMathOperator{\Tr}{Tr}

\DeclareMathOperator{\Inn}{Inn}

\newcommand{\longdownmapsto}{\rotatebox[origin=c]{-90}{$\longmapsto$}\mkern2mu}

\def\III{{I\!I\!I}}

\newcommand{\op}{\mathrm{op}}

\newcommand*\longhookrightarrow{%
    \ensuremath{\lhook\joinrel\relbar\joinrel\rightarrow}
}

\usepackage{xspace}

\makeatletter
\DeclareRobustCommand{\etc}{%
    \@ifnextchar{.}%
        {etc}%
        {etc.\@\xspace}%
}
\makeatother

\newcommand{\Cstar}{$C^\ast$\@\xspace}

\newcommand{\nsubset}{\not\subset}

\def\u1net{{\A_\RR}}

\theoremstyle{plain}
\newtheorem{theorem}{Theorem}[section]
\newtheorem{corollary}[theorem]{Corollary}
\newtheorem{lemma}[theorem]{Lemma}
\newtheorem{proposition}[theorem]{Proposition}

\theoremstyle{definition}
\newtheorem{definition}[theorem]{Definition}

\theoremstyle{remark}
\newtheorem{example}[theorem]{Example}
\newtheorem{remark}[theorem]{Remark}

\begin{document}

\title{Braided categories of endomorphisms as invariants for local quantum field theories}

\author{Luca Giorgetti, Karl-Henning Rehren \\ \small Institute for
Theoretical Physics \\ \small Georg-August-Universit\"at G\"ottingen
\\ \small {\tt giorgetti,rehren@theorie.physik.uni-goettingen.de}}

\maketitle

\begin{abstract}
We want to establish the ``braided action'' (defined in the paper) of
the DHR category on a universal environment algebra as a complete
invariant for completely rational chiral conformal quantum field
theories. The environment algebra can either be a single local
algebra, or the quasilocal algebra, both of which are
model-independent up to isomorphism. The DHR category as an abstract
structure is captured by finitely many data (superselection sectors,
fusion, and braiding), whereas its braided action encodes the full
dynamical information that distinguishes models with isomorphic DHR
categories. We show some geometric properties of
the \lq\lq duality pairing" between local algebras and the DHR
category which are valid in general (completely rational) chiral
CFTs. Under some additional assumptions whose status remains to be
settled, the braided action of its DHR category completely
classifies a (prime) CFT. The approach does not refer to the 
vacuum representation, or the knowledge of the vacuum state.
\end{abstract}

\section{Introduction}

In most approaches to quantum field theory (QFT) one
starts from a kinematical algebra (e.g., the equal-time canonical
commutation relations) and constructs the dynamics along with the
ground state (the vacuum). This state is represented, e.g., by the
path integral (after analytic continuation), which is notoriously
difficult to construct. It is well known that renormalization requires
a change of the original algebra along the way with the
construction. Once this is achieved, one extracts the (time-ordered)
correlation functions and scattering amplitudes.

In a recent approach based on the operator product expansion (OPE),
Holland and Hollands \cite{HoHo15} construct only the full interacting
quantum field algebra, whose coefficient functions turn out to be much
more regular at short distance than the vacuum correlation
functions. The construction 
of the algebra is in this approach \emph{well separated} from the
dynamical intricacies of the vacuum state, which must be constructed
in a second step.

This is very much in the spirit of the algebraic approach to quantum
field theory (AQFT) \cite{HaagBook}, which emphasizes the primacy of
the algebra of observables along with its local structure (its
subalgebras $\A(\O)$ of observables localized in spacetime regions
$\O$), and studies its many different representations of physical
interest. Among them, there is the vacuum representation,
distinguished by the existence of an invariant vacuum state
$\Omega$. The extraordinary features of this state are reflected in
the Bisognano-Wichmann property \cite{BiWi75}, \cite{BGL93},
\cite{Mun01} which asserts that its restriction to the algebra $\A(W)$
of observables in a wedge region $W$ is a KMS state for the boosts
subgroup preserving that wedge. This not only predicts remarkable
\lq\lq thermal features" of the well-known vacuum fluctuations,
including the Unruh effect \cite{Sew80}, \cite{BuVe14}, it also
allows to \emph{construct} the boost generator and the CPT operator
from just the data $(\A(W), \Omega)$, i.e., a single von Neumann
algebra and a state. Since the CPT operator differs from the
asymptotic free CPT operator by the scattering matrix \cite{Jos65}, it
carries most of the dynamical content of the QFT. 

The enormous amount of dynamical information encoded in the
quantum vacuum state is also witnessed by the following facts, which
may \lq\lq explain" why the construction of this state 
is bound to be so difficult. 

Borchers \cite{Bor92} has shown that a full (1+1)-dimensional QFT can
be constructed from a single algebra $\A(W)$, the vacuum state $\Omega$, and a
unitary positive-energy representation $U$ of the translations
subgroup, such that $U(x)\Omega=\Omega$ and $U(x)\A(W)U(x)^* \subset
\A(W)$ for $x\in W$. Using a pair of algebras and the vacuum state,
even the translations can be constructed \cite{Wie93-1}. This idea has
been extended to 3+1 dimensions in different ways, by Buchholz and
Summers \cite{BuSu93}, and by K\"ahler and Wiesbrock \cite{KaeWi01},
and to chiral conformal QFT by Guido, Longo and Wiesbrock \cite{GLW98}.

All these facts are instances of \emph{modular theory}, which
captures subtle functional analytic properties of faithful normal
states of von Neumann algebras. This theory is essentially trivial for
commutative algebras, and therefore none of these results has a
classical analogue.

In a nut-shell, all local algebras $\A(\O)$ of observables along with
the covariance, and hence the entire QFT, can be constructed out of one
or two given von Neumann algebras and the vacuum state.  

As an attempt to ``by-pass'' the difficult construction of the vacuum state,
we want to address the question, how far one can get \emph{without} 
knowledge of it, just given \lq\lq one or
two local von Neumann algebras", and which possibly more accessible
structure might be apt to substitute it?

Our input shall be the \emph{DHR category} \cite{DHR71} of the QFT
to be (re-) constructed, that controls the composition (\lq\lq fusion") and
permutation (\lq\lq braiding") of its positive energy
representations in terms of a \emph{unitary
braided tensor category} (UBTC) \footnote{It is actually even a \Cstar braided
tensor category, but the \Cstar property is automatic for 
rational UBTCs that we are going to deal with, see \cite[Lem.\ 3.2]{LoRo97}, \cite[Prop.\ 2.1]{Mue00}.}. 

In low dimensions, the DHR category may be regarded as a ``dual
substitute'' for global symmetries \cite{DoRo89}, \cite{DoRo90}, 
hence it encodes important but certainly not
complete information about the model. We shall see that its
\emph{braided action} on a model independent algebra, formulated in
Section \ref{sec:braction} as an
invariant for local nets,  
encodes more specific dynamical information. 

As abstract structures, UBTCs are quite easily accessible,
especially when they have only finitely many inequivalent irreducible
objects and finite-dimensional intertwiner spaces
(\emph{rational} QFT). In this case it suffices to know the 
fusion rules of the irreducible objects (superselection sectors), and
solve a finite number of algebraic relations to fix the admissible tensor structures and
braidings. E.g., the well-known fusion rules of the chiral Ising model
admit eight solutions, hence eight inequivalent UBTCs.

We want to explore to which extent the DHR category allows to
reconstruct the underlying QFT. The answer cannot be unique
because two QFTs may easily share the same DHR category
up to equivalence. E.g., by tensoring a QFT with
another one which has no nontrivial sectors (\lq\lq holomorphic CFT",
in the context of chiral conformal QFT) does not change its DHR
category. By invoking its braided action, however, the
  distinction is revealed, see
Section \ref{sec:primeCFTline}, and we offer a sufficient criterion to
exclude the presence of holomorphic factors. This criterion seems to
be the right one to grasp the information about localization
(left/right separation) of charges, hence \emph{dually} of
observables, out of the $\DHR$ braiding, in the sense of Proposition
\ref{prop:totalorderapts}. It is also a good candidate to be a necessary 
condition, in view of Proposition \ref{prop:holotensorsplit}.

We shall restrict ourselves to \emph{chiral conformal} QFTs, because
in this case \emph{complete rationality} \cite{KLM01} implies
\emph{non-degeneracy} of the $\DHR$ braiding, i.e., the $\DHR$ category has 
the abstract structure of a \emph{unitary modular tensor category} 
(UMTC). For our purpose, this means that 
the braiding of $\DHR$ endomorphisms encodes a sharp distinction between  
left and right. 
Our basic idea is to start with either the \emph{global} \Cstar-algebra 
$\A$ of quasilocal observables, or a single \emph{local} von Neumann 
algebra $\A(I_0)$ where $I_0$ is an arbitrarily fixed bounded interval 
of the line $\RR$ (or equivalently of the circle $\Sone$). 
The local picture is technically advantageous, but not essential, see Sections
\ref{sec:dualityrelations} and \ref{sec:localdualityrelations}. Indeed
neither $\A$, nor $\A(I_0)$, carry any specific information about the
models, by well-known results of \cite{Haa87}, \cite{Tak70},
and thus serve as a universal environment (``blanc
canvas'') to let the DHR category act on. 
 
Either locally or globally, relative commutants have a 
geometric interpretation both on half-intervals (strong additivity)
or half-lines (relative essential duality), see Proposition
\ref{prop:relessduality}. Also the structure of the two-interval
subfactor can be extended verbatim to a unital \Cstar-inclusion of
algebras in the real line picture, see Corollary
\ref{cor:reallinetwointerval}. Moreover the \emph{action} of the
$\DHR$ category on the observables behaves similarly locally or
globally: compare modularity with Proposition \ref{prop:dualofsergio},
and the duality relations between observables and endomorphisms
localizable in half-lines (Proposition \ref{prop:dualofwedge}) or
intervals (Proposition \ref{prop:dualoflocalcat}), either on $\RR$ or
confined in some fixed interval $I_0$. The latter proposition gives
also an affirmative answer (in the chiral conformal setting) to a conjecture
of S. Doplicher \cite{Dop82} (in (3+1)-dimensional theories), see Remark
\ref{rmk:dop82conj}.  

Our main tool to reconstruct the \emph{local substructure} of
the net are \emph{abstract points} of the braided action of the $\DHR$
category, see Section \ref{sec:apts}. 
The crucial observation is that the $\DHR$ category possesses, by its
very definition based on the underlying local structure, a
characteristic property: its braiding trivializes $\eps_{\rho,\sigma}
= \oneop$ whenever $\rho, \sigma$ 
are localizable in mutually left/right separated regions of the real
line. Since points are responsible
for left/right splittings of the line, this motivates our definition of abstract points as suitable
pairs of subalgebras that trivialize the braiding.

Using algebraic deformation techniques, 
abstract points can be carried wildly
far-away from the naive geometric picture of two half-interval
algebras, see Section \ref{sec:fuzzyapts}. We therefore need
to understand what is required to identify abstract point as geometric
points, up to unitary equivalence.
In Section
\ref{sec:Dedekind} we show a way of deriving the \emph{completeness}
of the braided action as an invariant for local nets, but on a
subclass of completely rational conformal nets which we call
\emph{prime conformal nets}, see Definition \ref{def:primeconfnet}. 
Primality of a conformal net rules out 
holomorphic and tensor products cases, and relies on the notion of
\emph{prime UMTC} due to \cite{Mue03}. In order to state the
classification result we actually need two further assumptions, see
Section \ref{sec:Dedekind}, hence the content of Proposition
\ref{prop:ptsandapts} is still an abstract recipe, as we do not know
which examples fit into the classification. Yet the recipe is quite
surprising and natural, in the sense that 
it is essentially based on two facts about completely rational nets:
the structure of the two-interval subfactor (\cite[Thm.\ 33]{KLM01})
and of the fixed points of the local $\DHR$ subcategories (Proposition
\ref{prop:dualoflocalcat}). 

In principle our techniques apply to general rational BTCs, in
particular to UMTCs, thanks
to realization results of \cite{HaYa00} by means of
endomorphisms. Hence solving the previous trivialization constraints
$\eps_{\rho,\sigma} = \oneop$ and then applying our machinery, can be
viewed as a possible way to \emph{realize} abstract UMTCs by means of
suitable, e.g., prime (see Definition \ref{def:primeconfnet}),
conformal nets via the $\DHR$ construction. We do not discuss this
\lq\lq exoticity" problem for abstract UMTCs in this work, and we
refer to \cite{Kaw15} for more explanations, and to \cite{Bis15} for a
systematic positive answer on the realization of Drinfeld doubles of
subfactors with index less than 4.

\section{Conformal nets and points on the line}\label{sec:CFTsandPtsonRR}

The purpose of this section is to collect structure
properties of QFT models that shall be used for the reconstruction of
local algebras from an action of the DHR category in later
sections. Although these results are well known (except Proposition
\ref{prop:relessduality}), it is worthwhile to exhibit them in due context.

In this work we deal with chiral conformal field theories
(\emph{chiral CFTs}) \lq\lq in one spacetime dimension", referring to
either of the two light-like coordinates $x^0\pm x^1$ in two
dimensions. By conformal covariance one can equivalently
consider theories on the real line $\RR$, or on the unit circle
$\Sone$. The latter can be regarded as a \lq\lq conformal closure" of
the line $\Sone \cong \overline \RR = \RR \cup \{\infty\}$ and the
points of the two sets can be put in bijective correspondence via the
Cayley map $x\in\RR \mapsto (x + i)(x - i)^{-1} \in
\Sone\smallsetminus \{1\}$.

Chiral CFTs are effectively described in the algebraic setting of
AQFT \cite{HaagBook}. An abundance of models of the
field-theoretic literature has been reformulated in this unifying
framework, giving access to model-independent insight and structure
analysis \cite{Reh15}.

In the following we adopt the \emph{real line picture} as more natural
for our purposes, in particular from a representation theoretical
point of view, cf.\ \cite{KLM01}. We describe chiral CFTs by means of
\emph{local conformal nets on the line} in the following sense, cf.\
\cite{FrJoe96}. Instead of \emph{points} of $\RR$ we have
\emph{bounded} intervals $I\subset \RR$, instead of local fields we
have \emph{local algebras} $\A(I)$. More precisely, let $\cI$ be the
family of non-empty open bounded intervals $I\subset \RR$ and notice
that $\cI$ is partially ordered by inclusion and directed. Consider a
complex separable Hilbert space $\Hilb$, the \emph{vacuum space}, and
to every $I\in\cI$ assign a von Neumann algebra $\A(I)=\A(I)''$
realized on $\Hilb$. The latter correspondence forms a \emph{net} of
algebras, which we denote by $\{\A\} = \{I\in\cI\mapsto\A(I)\}$. 

\begin{definition}\label{def:CFTline}
A net of von Neumann algebras $\{\A\} = \{I\in\cI\mapsto \A(I)\}$
realized on $\Hilb$ is a \textbf{local conformal net on the line} if
it fulfills: 
\begin{itemize} \itemsep0mm
\item \emph{Isotony}: if $I,J\in\cI$ and $I\subset J$ then $\A(I)
  \subset \A(J)$.  
\item \emph{Locality}: if $I,J\in\cI$ and $I\cap J = \emptyset$ then
  $\A(I)$ and $\A(J)$ elementwise commute. 
\item \emph{M\"obius covariance}: there is a strongly continuous
  unitary representation $U$ of the \emph{M\"obius group} $\Mob =
  PSL(2,\RR) = SL(2,\RR)/\{\pm \oneop\}$ on $\Hilb$, which acts
  covariantly on the net, i.e.  
$$U(g) \A(I) U(g)^* = \A(gI)$$
whenever $I\in\cI$, $g\in\Mob$ and $gI \in \cI$, we ask nothing otherwise.
\item \emph{Positivity of the (conformal) Hamiltonian}: the generator
  $H$ of the rotations subgroup of $\Mob$ is positive. 
\item \emph{Vacuum vector}: there exists a M\"obius invariant vector
  $\Omega\in\Hilb$, unique up to scalar multiples, and cyclic for $\{
  \A(I)$, $U(g) : I\in\cI, g\in\Mob \}$. 
\end{itemize}
A local conformal net on the line (in a \emph{vacuum sector}) is then
specified by a quadruple $(\{\A\},U,\Omega,\Hilb)$. 
\end{definition}

The following notion says when two local conformal nets are \lq\lq the
same", and is particularly useful for classification purposes. 

\begin{definition}\label{def:CFTisom}
Two local conformal nets on the line (in their vacuum sector) $\{\A\}$
and $\{\B\}$, or better $(\{\A\},U_{\A},\Omega_{\A},\Hilb_{\A})$ and
$(\{\B\},U_{\B},\Omega_{\B},\Hilb_{\B})$, are \textbf{isomorphic}, or
unitarily equivalent, if there exists a unitary operator
$W:\Hilb_{\A}\rightarrow\Hilb_{\B}$ which intertwines the two
quadruples, i.e., $W\A(I)W^* = \B(I)$ for all $I\in\cI$, $WU_\A(g)W^*
= U_\B(g)$ for all $g\in\Mob$ and $W\Omega_\A = \Omega_\B$. We write
$\{\A\}\cong\{\B\}$ for isomorphic nets. 
\end{definition}

Now starting from the local algebras of a net $\{\A\}$ as above, one
can define algebras for arbitrary regions $S\subset\RR$ as
follows. Define $\A(S)$ to be the von Neumann algebra, respectively
\Cstar-algebra, generated by all local algebras $\A(I)$ such that
$I\subset S$, depending on whether $S$ is a \emph{bounded},
respectively \emph{unbounded}, region of $\RR$. In the first case
notice that $\A(S)\subset\A(J)$ for a sufficiently big $J\in\cI$, in
the second case let $\cR(S):=\A(S)''$.

In this way we get the \emph{quasilocal \Cstar-algebra} $\A:=\A(\RR)$,
the algebras of \lq\lq space-like" complements of intervals $\A(I')$
where $I' := \RR\smallsetminus \overline I$, $I\in\cI$, the half-line
(\lq\lq wedge") algebras $\A(W)$ where $W\subset\RR$ is a non-empty
open half-line, left or right oriented. 

\begin{remark}
The latter distinction between norm and weak closure is not just
technical, it is essential to understand the structure of local nets
and their $\DHR$ representation theory. Assume Haag duality on $\RR$
(see below) and consider for instance $I\Subset J$, i.e., $\overline I
\subset J$ where $I,J\in\cI$. Then $I'\cap J = I_1\cup I_2$ and
$\A(I_1 \cup I_2) = \A(I_1) \vee \A(I_2) \subset \A(I)'\cap\A(J)$ is
the two-interval subfactor considered by \cite{KLM01}, and $\vee$ is a
short-hand notation for the von Neumann algebra generated. The
previous inclusion is proper in many examples, in particular $\DHR$
charge transporters from $I_1$ to $I_2$ do not belong to $\A(I_1 \cup
I_2)$.

On the other hand, take $I' = W_1\cup W_2$, $I\in\cI$ and observe that 
$$\A(W_1\cup W_2) = C^*\{\A(W_1)\cup\A(W_2)\} \subset \cR(W_1\cup W_2) = \A(W_1) \vee \A(W_2)$$ 
is by Haag duality on $\RR$ the inclusion $\A(I') \subset \A(I)'$,
again proper in general. In this case $\DHR$ charge transporters from
$W_1$ to $W_2$ are again not in $\A(W_1\cup W_2)$ but they belong to
the weak closure $\cR(W_1\cup W_2)$. Geometrically speaking,
half-lines $W_1$ and $W_2$ \lq\lq weakly touch at infinity'' and allow
charge transportation. 
\end{remark}

Chiral Rational CFTs (\emph{chiral RCFTs}) correspond, in the
algebraic setting, to a class of local conformal nets singled out by
the following additional conditions imposed on the local algebras, see
\cite{KLM01}, \cite{Mue10draft}. Throughout this paper we will
restrict to the completely rational case whenever representation
theoretical issues are concerned. 

\begin{definition}\label{def:RCFTline}
A local conformal net on the line $\{\A\}$, as in Definition
\ref{def:CFTline}, is called \textbf{completely rational} if the
following conditions are satisfied. 
\begin{itemize} \itemsep0mm
\item [(a)] \emph{Haag duality on $\RR$}: $\A(I')' = \A(I)$ for all $I\in\cI$.
\item [(b)] \emph{Split property}: for every $I,J\in\cI$, $I\Subset J$
  there exists a type $I$ factor $\cF$ such that $\A(I) \subset \cF
  \subset \A(J)$. 
\item [(c)] \emph{Finite index two-interval subfactor}: $\A(I_1\cup
  I_2) \subset \A(I)'\cap \A(J)$ has finite Jones index, where
  $I,J\in\cI$, $I\Subset J$ and $I'\cap J = I_1\cup I_2$ for
  $I_1,I_2\in\cI$. 
\end{itemize}
With conformal covariance, see \cite{GLW98}, condition (a) is equivalent to
\begin{itemize}
\item [(a)$'$] \emph{Strong additivity}: $\A(I_1\cup I_2) = \A(I)$
  where $I\in\cI$, $p\in I$ and $\{p\}'\cap I = I \smallsetminus \{p\}
  = I_1 \cup I_2$ for $I_1,I_2\in\cI$.  
\end{itemize}
\end{definition}

\begin{remark}
Conditions (a) and (b) strengthen the locality assumption on the net,
they are natural and fulfilled in many models. Condition (c) is the
characteristic feature of \lq\lq rational" theories, i.e., those with
finitely many superselection sectors.

Notice that complete rationality, in the conformal setting, is a local condition, i.e., can be checked inside one arbitrarily fixed local algebra.
\end{remark}

By conformal covariance, local conformal nets on the line $\{\A\}$, as in Definition \ref{def:CFTline}, can be uniquely extended to local conformal nets on the circle, see \cite{Lon2} for the precise definition of the latter. This fact is well known, cf.\ \cite{FrJoe96}, \cite{LoRe04}, \cite{LoWi11}, but contains some subtleties, see \cite[Sec. 1.2, 4.1]{GioPhD} for the details. In particular, denoted by $\{\tilde\A\}$ the extension, it can be shown that the two definitions one might give of weakly closed half-line algebras are the same, namely $\tilde\A(W) = \cR(W)$, and that in the Haag dual case (assumption (a)) the extension is algebraically determined by the formula $\tilde \A(I) = \A(I')'$. The correspondence $\{\A\}\mapsto\{\tilde\A\}$ is bijective up to isomorphism of nets in the sense of Definition \ref{def:CFTisom}.

As a consequence all the known properties of chiral conformal nets hold on the line as well, see, e.g., \cite{GaFr93}, \cite{GuLo96}, \cite{GLW98}. Notably the Reeh-Schlieder theorem, the Bisognano-Wichmann property, factoriality of the local algebras, additivity and essential duality $\cR(W)' = \cR(W')$. Moreover inclusions of local algebras $\A(I)\subset \A(J)$ for $I,J\in\cI$, $I\subset J$ are known to be \emph{normal} and \emph{conormal}, i.e., respectively
\begin{equation}\label{eq:normconorm}\A(I)^{cc} = \A(I),\quad\A(I) \vee \A(I)^c = \A(J)\end{equation}
where $\N^c := \N'\cap \M$ denotes the \emph{relative commutant} of the inclusion $\N\subset\M$ of von Neumann algebras.
The normality and conormality relations above do not
depend on the specific geometric position of $I$ inside $J$, nor on
Haag duality (assumption (a)). 

With the split property (assumption (b)) both the local algebras
$\A(I)$ for all $I\in\cI$ and the quasilocal algebra $\A$ are
\emph{canonical} objects, in the sense that they are universal
(independent of the specific model) up to spatial
isomorphism. The first as the unique injective (\lq\lq hyperfinite")
type $\III_1$ factor by \cite{Haa87}, the second by a general result of
\cite{Tak70}. In particular, they contain no specific information about
the models. Moreover \emph{locality} of the net is not needed neither
in \cite{Tak70} nor to apply the result of \cite{Haa87}. In the first
only isotony enters, for the second we know that Bisognano-Wichmann's
modular covariance holds regardless of locality \cite{DLR01}. 

The entire information about the chiral CFT is then encoded in the
inclusions and relative commutation relations among different local
algebras, i.e., in the \emph{local algebraic structure} of the
net. This statement is made precise by the next proposition due to
M. Weiner \cite{Wei11}, which says that the vacuum sector of a local
conformal net is \emph{uniquely determined} by its local algebraic
structure.
 
Let $\{\N_i\subset\M,\, i\in\cI\}$ and $\{\tilde
\N_i\subset\tilde\M,\, i\in\cI\}$ be two families of subfactors,
respectively in $\B(\Hilb)$ and $\B(\tilde\Hilb)$, indexed by the same
set of indices $\cI$. They are called \textbf{isomorphic} if there
exists a unitary operator $V:\Hilb\rightarrow\tilde\Hilb$ such that
$V\M V^* = \tilde\M$ and $V\N_i V^* = \tilde\N_i$ for all $i\in\cI$.  

\begin{proposition}\emph{\cite[Thm.\ 5.1]{Wei11}}.\label{prop:michiinvariant}
Let $\{\A\}$ be a local conformal net as above fulfilling the split property (assumption (b)). Then $\{\A\}$, or better $(\{\A\},U,\Omega,\Hilb)$, is completely determined up to isomorphism of nets, see Definition \ref{def:CFTisom}, by the isomorphism class of the local subfactors $\{\A(I)\subset\A(I_0),\, I\in\cI, I\subset I_0\}$ for any arbitrarily fixed interval $I_0\in\cI$.
\end{proposition}

In other words, the isomorphism class of the collection of local algebras is a \emph{complete invariant} for split local conformal nets.

With Haag duality on $\RR$ (assumption (a)), there is a geometric interpretation of the relative commutant and of the normality and conormality relations (\ref{eq:normconorm}) for inclusions of local algebras which arise for the choice of \emph{points}. Namely let $I\in\cI$, take $p\in I$ and let $\{p\}'\cap I = I \smallsetminus \{p\} = I_1 \cup I_2$, $I_1,I_2\in\cI$. The relative commutant of $\A(I_1) \subset \A(I)$ is then given by
\begin{equation}\label{eq:geomlocrelcomm}\A(I_1)^c := \A(I_1)' \cap \A(I) = \A(I_2).\end{equation}
It follows from conformal covariance, cf.\ \cite{GLW98}, that the relations (\ref{eq:geomlocrelcomm}) are actually equivalent to assumption (a).

Now a \textbf{point of an interval}, $p\in I$, is uniquely determined
by two intervals $I_1, I_2\in\cI$ as above, the \emph{relative 
complements} of $p$ in $I$. Algebraically, $p\in
I$ splits $\A(I)$ into a pair of commuting subalgebras 
\textbf{$\A(I_1),\A(I_2)\subset \A(I)$} which in the Haag dual case
are each other's relative commutants.

Similarly a \textbf{point of the line}, $p\in\RR$, is uniquely
determined by two half-lines $W_1, W_2\subset\RR$, the relative
complements of $p$ in $\RR$, and determines two \lq\lq global" unital
\Cstar-inclusions $\A(W_1),\A(W_2)\subset \A := \A(\RR)$. Our first
main structure result, see Proposition \ref{prop:relessduality}, shows
that the same geometric interpretation of relative commutants holds in
the global case. The proof is independent of assumption (a), but as a
technical tool we need to assume (b). Merging the standard terminology
of \lq\lq relative commutant" and  \lq\lq essential duality" for local
algebras we can call this property \emph{relative essential duality}. 

\begin{proposition}\label{prop:relessduality}
Let $\{\A\}$ be a local conformal net on the line as in Definition \ref{def:CFTline}, which fulfills the split property (assumption (b)). Consider the inclusion of unital \Cstar-algebras $\A(W)\subset \A$, where $W\subset\RR$ is a half-line, left or right oriented, then
$$\A(W)^c := \A(W)'\cap \A = \A(W')$$
where $W' = \RR\smallsetminus \overline W$ is the opposite half-line.
\end{proposition}

\begin{proof}
Observe first that $\A(W)'=\cR(W')$, hence the statement is equivalent
to $\A(W)=\cR(W)\cap \A$. This does not boil down to
essential duality $\cR(W)' = \cR(W')$, because typically $\A(W)\subset
\cR(W)$ is proper and $\cR(W)\nsubset\A$, see \cite[Sec. 1]{BGL93}.

By the split property we have that $\cR(W)$ is the injective factor of type $\III_1$ and the same holds for its commutant. Consider then a norm continuous conditional expectation 
$$E: \B(\Hilb) \rightarrow \cR(W)'$$
given by averaging over the adjoint action of the unitary group $G :=
\cU(\cR(W))$ of $\cR(W)$, equipped with the ultraweak topology or
equivalently with any of the other weak operator topologies.
 
Now, injectivity is equivalent to amenability of the unitary group, i.e., to the existence of a left invariant state (\lq\lq mean'') on the unital \Cstar-subalgebra $\C_{ru}(G)$ of right uniformly continuous functions in $L^\infty(G)$, see \cite{dlH79}, \cite{Pat92}. Similar to \cite{Arv74} one can define an integral $E(b) := \int_G \Ad_u(b)\du$ with respect to such a mean $m$, for every $b\in\B(\Hilb)$, as the unique element in $\B(\Hilb)$ such that
$$\langle\varphi,\int_G \Ad_u(b)\du\rangle = \int_G \langle\varphi, \Ad_u(b)\rangle \du\quad\forall \varphi\in\B(\Hilb)_*$$
where $\B(\Hilb)_*$ is the predual, and the r.h.s.\ is defined by the mean on functions
$$\int_G \langle\varphi, \Ad_u(b)\rangle \du = m(f_{\varphi,b}),\quad f_{\varphi,b}(u) := \langle\varphi, \Ad_u(b)\rangle.$$
One can easily see by formal computations that $E(b)u = uE(b)$ for all
$u\in G$ hence $E(b)\in \cR(W)'$, see also \cite[Lem.\ 1,
2]{dlH79}. Moreover, $E$ is a norm one projection \emph{onto}
$\cR(W)'$, i.e., $\|E(b)\| \leq \|b\|$ and $E(b) = b$ if
$b\in\cR(W)'$, hence a conditional expectation by
\cite{Tom57}. Observe that $E$ cannot be normal because $\cR(W)$ is
type $\III$, see \cite[Ex.\ IX.4]{Tak2}.

The next step is to show that $E$ preserves the local structure of the net, i.e., maps local algebras into local algebras and $\A$ into itself. So take a bounded interval $I$ containing the origin of $W$, we want to show that
$$E:\A(I) \rightarrow \A(I) \cap \cR(W)'.$$
First, assume in addition that Haag duality on $\RR$ holds. Take $a\in\A(I)$ and $\A(I) = \A(I')' = (\cR(W_1) \vee \cR(W_2))'$ where $I' = W_1 \cup W_2$ and $W_1,W_2$ are half-lines. If for instance $W_2 \subset W$, then every $x\in \cR(W_2)$ commutes with $E(a)\in\cR(W)'$. Take now any $y\in\cR(W_1)\subset \cR(W')$, then
$$E(a) y = \int_G \Ad_u(a) y \du = \int_G y \Ad_u(a) \du = y E(a)$$
because $uy = yu$, $u\in \cR(W)$ and $ay = ya$, $a\in\A(I)$ by
locality. Hence $E(a)$ commutes with $\cR(W_2)$ and with
  $\cR(W_1)$, and we can conclude that $E(a)\in\A(I)$.
 
In general, a more refined and purely algebraic argument \cite[Lem.\ 2 (iii)]{dlH79} shows directly that $E(a) \in \A(I) \vee \cR(W)$ which coincides with $\cR(W_1')$ by additivity, hence $E(a) \in \cR(W_1' \cap W')$ where $W_1' \cap W' = I\cap W'\in \cI$ and
$$E: \A(I) \rightarrow \A(I\cap W') = \A(I) \cap \cR(W)'.$$
Exhausting $\RR$ with a sequence of intervals $I_n$ containing the origin of $W$, by norm continuity of $E$ we get $E : \A \rightarrow \A$ and 
$$C^*\{ \bigcup_n \A(I_n\cap W') \} = E(\A) = \A(W)^c.$$
But also $C^*\{ \bigcup_n \A(I_n\cap W') \} = \A(W')$, hence $\A(W)^c = \A(W')$ follows.
\end{proof}

\begin{remark}
The techniques employed here are similar to those used in \cite[Sec. 5]{Dop82}. There, however, local algebras $\A(I)$ are considered instead of half-line algebras and one does not need additivity nor essential duality to show that conditional expectations on $\A(I)'$ preserve the local substructure of $\A$.
\end{remark}

As a consequence of Proposition \ref{prop:relessduality}, assuming the split property we can take the relative commutant of the inclusion $\A(W')\subset\A(W)^c\subset\A(W)'$ and obtain 
\begin{equation}\label{eq:globalnormalityhalfline}\A(W) = \A(W)^{cc} = \cR(W) \cap \A\end{equation}
where the relative commutants refer to the inclusions
$\A(W)\subset\A$.
 
This is similar to the case of local algebras $\A(I)\subset\A$, $I\in\cI$ if we assume Haag duality on $\RR$, indeed
\begin{equation}\label{eq:globalnormalityinterval}\A(I) = \A(I)^{cc}\end{equation}
follows by taking relative commutants of the inclusion
$\A(I')\subset\A(I)^c\subset\A(I)'$, cf.\ \cite[Sec. V]{DHR69I}. The
relations (\ref{eq:globalnormalityhalfline}) and
(\ref{eq:globalnormalityinterval}) are a global version of the
normality relations (\ref{eq:normconorm}) encountered before.
 
Heuristically speaking, we regard \emph{normality} as an algebraic fingerprint of \emph{connectedness} in the following sense. Algebras associated to intervals $\A(I)$ or half-lines $\A(W)$ are \lq\lq connected", relative commutants $\A(I)^c$ are also \lq\lq connected" in a broader sense, e.g., on the circle, because $\A(I)^c = \A(I)^{ccc}$ always holds. On the other hand, algebras $\A(S)\subset\A$ associated to disconnected regions, e.g., $S=I'$, $I\in\cI$, need not be normal. Indeed, assuming (a), the inclusion 
\begin{equation}\label{eq:globaltwointervalinclusion}\A(I') \subset \A(I')^{cc} = \A(I)^c\end{equation}
is proper in many examples, see Corollary \ref{cor:reallinetwointerval}. In the case of \emph{holomorphic} nets there is no algebraic distinction (in the sense of normality relations) between \lq\lq connected" and \lq\lq disconnected" regions at the level of nets, cf.\ \cite{ReTe13} for an explicit isomorphism between interval and two-interval algebras in the case of graded-local Fermi nets. Notice that the unital \Cstar-inclusion (\ref{eq:globaltwointervalinclusion}) is a \lq\lq global" version of the \emph{two-interval subfactor} $\A(I_1\cup I_2) \subset \A(I_1\cup I_2)^{cc} = \A(I)^c$ considered by \cite{KLM01}, where relative commutants are taken in $\A(J)$ for $I\Subset J$, $I'\cap J = I_1 \cup I_2$. Indeed $((\A(I_1)\vee\A(I_2))'\cap\A(J))'\cap\A(J) = (\A(I_1)' \cap \A(I_1 \cup I))'\cap\A(J) = \A(I)'\cap\A(J)$.

In the following we shall concentrate on \emph{local conformal nets on
  the line} $\{\A\}$, see Definition \ref{def:CFTline}, which are in
addition \emph{completely rational}, as in Definition
\ref{def:RCFTline}. In this case we know by \cite[Cor.\ 37]{KLM01} that
the category of finitely reducible \textbf{DHR representations} of
the net, denoted by $\DHR\{\A\}$, has the abstract structure of a
\emph{unitary modular tensor category} (UMTC).
Referring to \cite{DHR71}, \cite{FRS92}, \cite{BKLR15II}, \cite{Mue12},
\cite{EGNO15} for the relevant definitions and further details, we just recall that DHR
representations of a local quantum field theory satisfying Haag
duality can be described in terms of \textbf{DHR endomorphisms} of the
quasilocal algebra $\A$, which enjoy covariance, localizability and
transportability properties. They are the objects of the \Cstar
tensor category $\DHR\{\A\}$, and their intertwiners are the morphisms. The fusion product of representations is 
defined through the composition of DHR endomorphisms (the monoidal
product of $\DHR\{\A\}$), which is commutative up to unitary equivalence. 
The unitary equivalence between $\rho\circ\sigma$ and
$\sigma\circ\rho$ is given by the DHR braiding  
$$\eps_{\rho,\sigma}= (v^*\times u^*)\cdot(u \times v) = \sigma(u^*)v^* u \rho(v)\in\Hom(\rho\,\sigma,\sigma\rho)$$
where $u\in\Hom(\rho,\hat\rho)$ and $v\in\Hom(\sigma,\hat\sigma)$ are
unitary charge transporters to equivalent auxiliary DHR endomorphisms
$\hat\rho$, $\hat\sigma$, such that $\hat\rho$ is localizable to the
space-like left of $\hat\sigma$ \footnote{In \cite{FRS92} the opposite right/left convention is
adopted for the $\DHR$ braiding; this is related to a different
convention for the Cayley map given at the beginning of this section.}. 
The unitary braiding thus defined
does not depend on the specific choice of the auxiliary endomorphisms
$\hat\rho$, $\hat\sigma$, and of the charge transporters $u$ and $v$,
and satisfies the naturality axiom, thus turning $\DHR\{\A\}$ into a unitary braided 
tensor category (UBTC). By the definition, if $\rho$ is localizable to the
space-like left of $\sigma$, one may choose $u=v=\oneop$, hence 
$$\eps_{\rho,\sigma}=\oneop.$$
UMTCs are a particular class of UBTCs having irreducible tensor unit,
finitely many inequivalent 
irreducible objects, conjugate objects and non-degenerate braiding
(\textbf{modularity}).
 
The latter is the essentially new feature of $\DHR$ categories arising
in low-dimensional models. Moreover, the key ingredient in the proof
of modularity is the discovery of a deep connection between the
algebraic structure of the net and the structure of its representation
category. More precisely, the two-interval subfactor
\cite[Thm.\ 33]{KLM01} is a Longo-Rehren subfactor
\cite[Prop.\ 4.10]{LoRe95} and is uniquely determined up to isomorphism
by the tensor structure of the category (forgetting the braiding), see
\cite[Cor.\ 35]{KLM01}. Hence the \textbf{DHR braiding} can be seen as
an additional ingredient whose
definition requires, in the low-dimensional case, the choice of a
point (irrespectively of its position) in order to separate the 
localization of DHR endomorphisms.

We close the section by mentioning that complete rationality is
realized by several models: Wess-Zumino-Witten $SU(N)$-currents
\cite{Was98}, Virasoro nets with central charge $c<1$ \cite{Car04},
\cite{KaLo04}, lattice models \cite{DoXu06}, \cite{Bis12}, the
Moonshine vertex operator algebra \cite{KaLo06}. Further candidates
come from more general loop groups \cite{GaFr93} and vertex operator
algebras \cite{CKLW15}. Moreover, complete rationality passes to
tensor products \cite{KLM01}, group-fixed points \cite{Xu00-2}, finite
index extensions and finite index subnets \cite{Lon03}.

\section{Braided actions of DHR categories}\label{sec:braction}

The motivation of our work is the following: in the variety of
completely rational models, one can easily find non-isomorphic ones,
see Definition \ref{def:CFTisom}, having equivalent $\DHR$ categories
in the sense of abstract UBTCs, see \cite[Def.\ 8.1.7, Rmk.\
9.4.7]{EGNO15}. Examples of this can be constructed by looking at
completely rational \textbf{holomorphic nets}, i.e., nets with only
one irreducible $\DHR$ sector: the vacuum. In this case the $\DHR$
category coincides with $\catVec$, the category of finite-dimensional
complex vector spaces, up to unitary braided tensor equivalence. Take
now a completely rational conformal net $\{\A\}$ and tensor it with a
nontrivial holomorphic net $\{\Aholo\}$, then \footnote{Here $\simeq$ 
denotes UBTC equivalence and $\boxtimes$ is the Deligne product 
(the \lq\lq tensor product" in the category of semisimple linear categories).}
$$\DHR\{\A\otimes\Aholo\} \simeq \DHR\{\A\}\boxtimes\DHR\{\Aholo\} \simeq \DHR\{\A\}$$
but $\{\A\}\ncong\{\A\otimes\Aholo\}$, because tensoring
with nontrivial holomorphic nets increases the central charge by a multiple of 8. 
Hence the UBTC equivalence class of the $\DHR$ category is \emph{not} a 
complete invariant for nets, i.e., the correspondence between completely 
rational conformal nets (up to isomorphism) and their $\DHR$ categories (up to UBTC equivalence)
\begin{equation}\label{eq:corresp}\{\A\} \mapsto \DHR\{\A\}\end{equation}
is not one-to-one. We might replace equivalence of categories with the
much stronger notion of isomorphism of categories, see \cite{Mac98},
but this is not what we want to do. Instead we consider the
\textbf{action} of the $\DHR$ category on the net as additional
structure, i.e., consider its realization as a \emph{braided tensor
  category of endomorphisms of the net}. For technical reasons, we
look at the action on a local algebra rather than the \lq\lq global"
defining action $\DHR\{\A\}\subset\End(\A)$ on the quasilocal
algebra. Namely, fix an arbitrary interval $I_0\in\cI$ and consider
the \lq\lq local" full subcategory $\DHR^{I_0}\{\A\}\subset\DHR\{\A\}$
whose objects are the $\DHR$ endomorphisms $\rho$ localizable in
$I_0$, i.e., $\rho_{\restriction\A({I_0}')} =
\id_{\restriction\A({I_0}')}$.
 
Notice that the inclusion functor in this case is also an equivalence, i.e., essentially surjective in addition
\begin{equation}\label{eq:nominimalloc}\DHR^{I_0}\{\A\} \simeq \DHR\{\A\}\end{equation}
because $I_0$ is open and there is by definition (and by M\"obius covariance) no minimal localization length. Considering the action on local algebras means considering the \emph{restriction functor} $\rho\mapsto\rho_{\restriction\A(I_0)}$
\begin{equation}\label{eq:resfun}\DHR^{I_0}\{\A\} \hookrightarrow\End(\A(I_0))\end{equation}
which is well-defined, \emph{strict tensor} and \emph{faithful} by Haag duality on $\RR$.
Recall that the arrows of the endomorphism category on the right hand side are defined as 
$$\Hom_{\End(\A(I_0))}(\hat\rho,\hat\sigma) := \big\{t\in\A(I_0) : t\hat\rho(a) = \hat\sigma(a)t\,,\, a\in\A(I_0)\big\}$$
where $\hat\rho, \hat\sigma\in\End(\A(I_0))$. With conformal symmetry \cite{GuLo96} have shown that the restriction functor is also \emph{full} (i.e., local intertwiners are global), hence an \emph{embedding} of categories.
The restriction functor is by no means essentially surjective, i.e.,
not every (finite index) endomorphism of the injective type $\III_1$
factor $\A(I_0)$ is realized by $\DHR$ endomorphisms of
$\{\A\}$. But it has \emph{replete image}, i.e., it is closed under unitary isomorphism classes in $\End(\A(I_0))$.

The first interesting point concerning the embedding (\ref{eq:resfun}) is the following

\begin{remark}\label{rmk:popathm}
Forgetting the braiding, the remaining \emph{abstract} structure of
$\DHR^{I_0}\{\A\}$ is the one of a \emph{unitary fusion tensor
  category} (UFTC). Functors between unitary categories
(or *-categories) will always be assumed to preserve the
*-structure. A result of Popa \cite{Pop95} states that an embedding
$\C\hookrightarrow \End(\M)$ as above, where $\C$ is a UFTC and $\M$
is the unique injective type $\III_1$ factor, is canonical in the
following sense. 
Take two equivalent UFTCs realized as endomorphisms of injective type $\III_1$ factors $\C\subset\End(\M)$ and $\D\subset\End(\N)$ where we can assume $\M,\N \subset \B(\Hilb)$. By \cite[Cor.\ 6.11]{Pop95}, see also \cite[Cor.\ 35]{KLM01}, there exists a spatial isomorphism $\Ad_U:\M\rightarrow\N$ where $U$ is unitary in $\B(\Hilb)$ which implements an equivalence $\C\simeq\D$ as follows
\begin{equation}\label{eq:outerconjugacy}\hat\rho_i\mapsto\Ad_{U} \circ\, \hat\rho_i \circ\Ad_{U^*} \simeq  \hat\sigma_i\end{equation}
for all $i = 0,\ldots,n$ where $\{\hat\rho_0,\ldots,\hat\rho_n\}$ and
$\{\hat\sigma_0,\ldots,\hat\sigma_n\}$ are generating sets for $\C$
and $\D$ respectively and $\simeq$ stands for unitary isomorphism in 
$\End(\N)$.

If both embeddings are \emph{replete} as in (\ref{eq:resfun}), we can
extend the  equivalence (\ref{eq:outerconjugacy}) to an isomorphism of
categories $\C\cong\D$ and every $\hat\sigma\in\D$ can be written as 
$$\hat\sigma = \Ad_{U} \circ\, \hat{\rho} \circ\Ad_{U^*} =: {^U} \hat\rho$$
for a unique $\hat\rho\in\C$, moreover $t\mapsto \Ad_U(t) =: {^U} t$
gives a *-linear bijection of the Hom-spaces $\Ad_U:
Hom(\hat\rho_i,\hat\rho_j) \rightarrow Hom({^U} \hat\rho_i,
{^U}\hat\rho_j)$. This isomorphism is manifestly strict tensor. 
\end{remark}

Take two nets $\{\A\}$, $\{\B\}$ and consider as in (\ref{eq:resfun}) the replete embeddings of the respective $\DHR$ categories
$$\DHR^{I_0}\{\A\} \hookrightarrow\End(\A(I_0)) , \quad \DHR^{I_0}\{\B\} \hookrightarrow\End(\B(I_0))$$
for some fixed interval $I_0\in\cI$. As we said, it may happen that $\DHR\{\A\} \simeq \DHR\{\B\}$ as UBTCs, hence as UFTCs forgetting the braiding. By Remark \ref{rmk:popathm}, there is a spatial isomorphism $\Ad_U:\A(I_0) \rightarrow \B(I_0)$ which implements a strict tensor isomorphism between the images of the two restrictions, hence between the respective local $\DHR$ subcategories.

However, the latter isomorphism
$F_U:\DHR^{I_0}\{\A\}\rightarrow\DHR^{I_0}\{\B\}$ need \emph{not} preserve the braidings
$$\eps^\A_{\rho_1,\rho_2} = v_2^*\times u_1^*\cdot u_1\times v_2 = \rho_2(u_1^*)v_2^* u_1 \rho_1(v_2)
\in\Hom_{\DHR\{\A\}}(\rho_1\rho_2,\rho_2\rho_1)$$
where $\rho_1,\rho_2\in\DHR^{I_0}\{\A\}$ and $u_1,v_2$ are unitaries in $\A(I_0)$ such that $\Ad_{u_1}\rho_1$ is localizable left to $\Ad_{v_2}\rho_2$ inside $I_0$. Indeed
$$F_U(\eps^\A_{\rho_1,\rho_2}) = \Ad_U(\rho_2(u_1^*)v_2^* u_1 \rho_1(v_2)) = F_U(v_2^*)\times F_U(u_1^*)\cdot F_U(u_1)\times F_U(v_2)$$
is in the correct intertwiner space
$$F_U(\eps^\A_{\rho_1,\rho_2}) \in \Hom_{\DHR\{\B\}}(F_U(\rho_1)F_U(\rho_2),F_U(\rho_2)F_U(\rho_1))$$
but can be $F_U(\eps^\A_{\rho_1,\rho_2}) \neq \eps^\B_{F_U(\rho_1),F_U(\rho_2)}$ because, for instance, $F_U(u_1), F_U(v_2)$ need not be charge transporters which take the respective endomorphisms one left to the other inside $I_0$.

Take now two \emph{isomorphic} nets $\{\A\}$,
$\{\B\}$ (see Definition \ref{def:CFTisom}). Then there is a unitary $W$ which implements spatial isomorphisms $\Ad_W:\A(I) \rightarrow \B(I)$ for \emph{every} $I\in\cI$, hence for $I_0$ and all of its subintervals. The resulting strict tensor isomorphism $F_W:\DHR^{I_0}\{\A\}\rightarrow\DHR^{I_0}\{\B\}$ defined on objects as $\rho \mapsto \Ad_W \circ\, \rho \circ \Ad_{W^*}$ is \emph{braided} in addition. Indeed $F_W$ respects the localization regions of the $\DHR$ endomorphisms, by definition, hence $F_W(\eps^\A_{\rho_1,\rho_2}) = \eps^\B_{F_W(\rho_1),F_W(\rho_2)}$. More generally

\begin{definition}\label{def:braction}
Let $\C$ be an abstract strict UMTC and $\M$ a von Neumann factor. A strict tensor replete embedding 
$$G: \C \hookrightarrow \End(\M)$$
will be called a \textbf{braided action} of $\C$ on $\M$.
\end{definition}

\begin{remark}
The previous notion is purely \emph{tensor} categorical, indeed the
category $\End(\M)$ is an enormous object which does not have a \lq\lq
global" braiding. However any braided action can be promoted to an
actual \emph{braided} functor 
by endowing the (replete tensor) image $G(\C) \subset \End(\M)$ with the braiding $\hat\eps_{G(\rho), G(\sigma)} := G(\eps_{\rho,\sigma})$.
Our terminology is motivated by the importance of the realization of
$\C$ as a \emph{braided} tensor category of endomorphism of $\M$, see
Definition \ref{def:bractionisom} below for the precise formulation of
this statement. The endomorphisms in the range of the
embedding have automatically finite index. Moreover if $\M$ is type
$\III$, they are automatically normal and injective (unital). 
\end{remark}

In our case at hand, $\C := \DHR^{I_0}\{\A\}$ for some fixed
$I_0\in\cI$ and the \textbf{braided action of the DHR category},
remember the equivalence (\ref{eq:nominimalloc}), on $\M_0 := \A(I_0)$
is given by the \emph{restriction functor} (\ref{eq:resfun}).

\begin{definition}\label{def:bractionisom}
Let $\C$, $\D$ be two abstract strict UMTCs and
$\M$, $\N$ two von Neumann 
factors. Two braided actions $G_1:\C \hookrightarrow \End(\M)$ and
$G_2:\D \hookrightarrow \End(\N)$ will be called \textbf{isomorphic}
if there is a spatial isomorphism $\Ad_U:\M\rightarrow \N$ implementing a strict tensor isomorphism between
the respective images which is also braided. Equivalently, the unique strict tensor isomorphism
$F_U: \C \rightarrow \D$ which makes the following diagram commute
\[
\begin{array}{ccr}
  \hskip3mm\C & \stackrel{G_1}{\longhookrightarrow} & \End(\M) \\ 
\scriptstyle{F_U}\Big\downarrow && \Big\downarrow\scriptstyle{\Ad_U} \\ 
  \hskip3mm\D & \stackrel{G_2}\longhookrightarrow  &  \End(\N)
\end{array}
\]
is in addition a UBTC isomorphism.
\end{definition}

Take two nets $\{\A\}$, $\{\B\}$, their respective $\DHR$ categories
together with their braided actions respectively on $\A(I_0)$,
$\B(I_0)$ for some fixed $I_0$. Clearly from the previous discussion,
if $\{\A\}$ and $\{\B\}$ are isomorphic nets (see Definition
\ref{def:CFTisom}) then $\DHR^{I_0}\{\A\}$ and $\DHR^{I_0}\{\B\}$ have
isomorphic braided actions (see Definition \ref{def:bractionisom})
hence we have an \emph{invariant}.

Remarkably, the situation described in Definition \ref{def:braction}
is general for UMTCs, in the sense that 
every abstract UMTC $\C$ admits a braided action on the injective type $\III_1$ factor $\M$.

\begin{remark}\label{rmk:yamembedd}
As in Remark \ref{rmk:popathm}, we drop the braiding on $\C$ and
consider its UFTC structure first. Without loss of generality, i.e.,
up to a (non-strict) tensor equivalence \cite[Thm.\ 1, \S XI.3]{Mac98},
we can assume that $\C$ is strict. Relying on a deep result of
\cite{HaYa00}, we know that the presence of conjugates (rigidity) and
the \Cstar-structure guarantee the existence of a (non-strict) tensor
embedding $G:\C\hookrightarrow\End(\M)$, where $\M$ is the unique
injective type $\III_1$ factor. Now the image of $\C$ in $\End(\M)$
can be endowed with the braiding which promotes $G$ to a braided
embedding, taking care of the nontrivial multiplicativity constraints
of the functors, and can be completed to a UMTC $\hat \C$ realized and
replete in $\End(\M)$, which is equivalent to $\C$ as an abstract
UMTC. The inclusion functor gives then a braided action of $\hat\C$ on
$\M$ in the strong sense employed in Definition
\ref{def:braction}. Similarly to Remark \ref{rmk:popathm}
but in this more general context, the (non-strict) tensor embedding
$G:\C\hookrightarrow\End(\M)$ of a UFTC $\C$ is also expected to be
unique (in a suitable sense, cf.\ \cite[Conj.\ 3.6]{HePe15}). 
\end{remark}

\section{Duality relations}\label{sec:dualityrelations}

Motivated by \cite{Dop82} we consider the duality pairing 
\begin{equation}\label{eq:dualpar}\A\,\stackrel{\perp}{\longleftrightarrow}\,\DHR\{\A\}\end{equation}
between the $\DHR$ category and the algebra $\A$ of quasilocal observables of a given (Haag dual) local conformal net $\{\A\}$, defined by the action 
$(a,\rho) \mapsto \rho(a)$.

\begin{definition}
Given a unital \Cstar-subalgebra $\N\subset\A$ we define its \textbf{dual} as
$$\N^\perp := \big\{\rho\in\DHR\{\A\} : \rho(n) = n,\, n\in \N\big\}$$
and $\Hom_{\N^\perp}(\rho,\sigma) := \Hom_{\DHR\{\A\}}(\rho,\sigma)$ for every $\rho,\sigma\in\N^\perp$. In other words, $\N^\perp \subset \DHR\{\A\}$ is a full subcategory, i.e., specified by its objects only. 
\end{definition}

$\N^\perp$ is automatically a unital tensor category of endomorphisms of $\A$. Conversely 

\begin{definition}
Given a unital tensor full subcategory $\C\subset \DHR\{\A\}$ we define its \textbf{dual} as
$$\C^\perp := \big\{a\in \A : \sigma(a) = a,\,\sigma\in\C\big\}.$$
\end{definition}

$\C^\perp$ is automatically a unital \Cstar-subalgebra of $\A$. We have the following

\begin{proposition}\label{prop:dualofwedge}
Let $\{\A\}$ be a local conformal net on the line fulfilling in addition Haag duality on $\RR$ (assumption (a)). Take $\A(W)\subset\A$ where $W\subset\RR$ is a half-line, left or right oriented, then
$$\A(W)^\perp = \DHR^{W'}\{\A\}$$
where $\DHR^{W'}\{\A\}$ is the full subcategory of $\DHR\{\A\}$ whose objects are the endomorphisms localizable in the half-line $W'$, opposite to $W$. 
\end{proposition}

\begin{proof}
One inclusion is trivial, the other follows from the definition of $\DHR$ localizability of endomorphisms and norm continuity.
\end{proof}

Combining Proposition \ref{prop:relessduality} and
\ref{prop:dualofwedge} we obtain

\begin{corollary}
Let $\{\A\}$ be a local conformal net on the line fulfilling Haag duality on $\RR$ (assumption (a)) and the split property (assumption (b)). Then
${\A(W)^c}^\perp = \DHR^{W}\{\A\}$ for every half-line $W\subset\RR$, left or right oriented. In particular
$${\A(W)}^\perp \simeq \DHR\{\A\} \simeq {\A(W)^c}^\perp$$
as UBTCs. 
\end{corollary}

Also, by definition, we have trivial braiding operators
\begin{equation}\label{eq:QFTbraiding}\eps_{\rho\,\sigma} = \oneop\end{equation}
whenever $\rho\in\DHR^{W}\{\A\}$, $\sigma\in\DHR^{W'}\{\A\}$ and $W$
is a \emph{left} half-line, hence $W'$ a \emph{right} half-line.
Equation (\ref{eq:QFTbraiding}) is the characteristic feature of the
$\DHR$ braiding coming from spacetime localization of charges in
QFT. An abstract UBTCs need not have this kind of trivialization
property for braiding operators at all. 

The situation is different for local algebras $\A(I)\subset\A$,
$I\in\cI$, as shown by Doplicher in \cite[Prop.\ 2.3]{Dop82} with the
split property (assumption (b)):  

\begin{proposition}\label{prop:dualofsergio} \emph{\cite{Dop82}.} 
Let $\{\A\}$ be a local conformal net on the line fulfilling in addition assumptions (a) and (b), then
$${\A(I)^c}^\perp = \langle\Inn^{I}\{\A\}\rangle_\oplus$$
for every $I\in\cI$, where $\Inn^{I}\{\A\}$ is the full subcategory of $\DHR\{\A\}$ whose objects are the inner automorphisms localizable in $I$ and $\langle - \rangle_\oplus$ denotes the completion under (finite) direct sums in $\A(I)$, i.e., the inner endomorphisms localizable in $I$.
\end{proposition}

In particular,
\begin{equation}\A(I)^\perp \simeq \DHR\{\A\} , \quad {\A(I)^c}^\perp \simeq \catVec.\end{equation}

\begin{remark}
The previous proposition has a deep insight in the theory of $\DHR$ superselection sectors in any spacetime dimension, see also \cite[Lem.\ III-1 (erratum)]{Bor65}, \cite[Sec. V]{DHR69I}, \cite[Sec. 1.9]{Rob11II} and discussions therein. Notice also that the proof in \cite{Dop82} is formulated in 3+1 dimensions and holds in the case of Abelian gauge symmetry, i.e., $\DHR$ automorphisms only. See \cite[Prop.\ 4.2]{Mue99} for the adaptation to the general case, and \cite{Dri79} for related arguments. Notice also that by definition $\DHR^{I}\{\A\} = \A({I}')^\perp$. 
\end{remark}

Furthermore, using now all the assumptions of complete rationality (a), (b), (c), we can prove our second main structure result

\begin{proposition}\label{prop:dualoflocalcat}
Let $\{\A\}$ be a completely rational conformal net on the line, then
$${\DHR^{I}\{\A\}}^\perp = \A({I}')$$
for every $I\in\cI$.
\end{proposition}

\begin{proof}
$(\supset)$: trivial by definition of $\DHR$ localization.

$(\subset)$: take $a\in\A$ such that $\rho(a)=a$ for all $\rho\in\DHR^{I}\{\A\}$. It follows easily that $a\in\A(I)^c = \A(I)'\cap \A$ by using inner automorphisms localizable in $I$, the task is to show that $a\in\A({I}')$. We divide the proof into three steps.

We first assume that (i) $a\in\A_{\rm loc}$, i.e., 
$a\in\A(K)$ for some sufficiently big interval $I\Subset K$ and that
(ii) all $\DHR$ endomorphisms have dimension $d_\rho = 1$
(pointed category case).

Then the inclusion $\A({I}')\subset\A(I)^c$ is locally the two-interval subfactor $\A(I_1\cup I_2)\subset \A(I)'\cap\A(K) = \A(I)^c$ where ${I}'\cap K = I_1\cup I_2$ and $I_1,I_2\in\cI$. Hence $a\in\A(I)^c$ has a unique \lq\lq harmonic'' expansion \cite[Eq.\ (4.10)]{LoRe95}
\begin{equation}\label{eq:harmonicexpansion} a = \sum_{i=0,\ldots,n} a_i \overline R_i \end{equation}
where $a_i\in\A(I_1\cup I_2)$ are uniquely determined coefficients and $\overline R_i\in\A(I)^c$ are (fixed) generators of the extension. The computation of this extension is the core of \cite{KLM01}. The extension has finite index by assumption (c) and the generators are uniquely determined, up to multiplication with elements of $\A(I_1\cup I_2)$, by the $\DHR$ category of $\{\A\}$. Indeed 
$$\overline R_i\in\Hom_{\DHR\{A(I)\}}(\id,\rho^1_i \overline{\rho}^2_i)$$
are solutions of the conjugate equations \cite[Sec. 2]{LoRo97} for the $i$-th sector $[\rho_i]$ where $\rho^1_i$ is localizable in $I_1$ and $\overline \rho^2_i$ is localizable in $I_2$, and $n$ is the number of $\DHR$ sectors of the theory different from the vacuum $[\rho_0] = [\id]$. By Frobenius reciprocity \cite[Lem.\ 2.1]{LoRo97} and up to multiplication with elements of $\A(I_1\cup I_2)$, the generators $\overline R_i$ can be thought as unitary $[\rho_i]$-charge transporters from $I_2$ to $I_1$, equivalently as unitary $[\overline\rho_i]$-charge transporters from $I_1$ to $I_2$. By assumption, for all $\rho\in \DHR^{I}\{\A\}$ we have
$$a = \sum_i a_i \overline R_i = \rho(a) = \sum_i a_i \rho(\overline R_i)$$
To fix ideas, from now on we assume $I_1$ left to $I$ and $I_2$ right to $I$.
By naturality and tensoriality of the braiding, see \cite[Lem.\ 2.6]{DHR71}, \cite[Sec. 2.2]{FRS92}, we have
$$\eps_{\rho^1_i, \rho}\, \rho^1_i(\eps_{\overline\rho^2_i, \rho})\,\overline R_i = \rho(\overline R_i)$$
which reduces to
$$\rho(\overline R_i) = \eps_{\overline\rho^2_i,\rho}\,\overline R_i$$
because of the respective localization properties of the endomorphisms. In this special case we have $\eps_{\overline\rho_i^2,\rho} = \lambda_{\overline\rho_i,\rho}\oneop$ where $\lambda_{\overline\rho_i,\rho}\in\mathbb{T}$ is a complex phase, hence $a_i\,\eps_{\overline\rho^2_i,\rho} \in \A(I_1\cup I_2)$ and by uniqueness of the previous expansion, if $a_i \neq 0$ we must have $\eps_{\overline\rho^2_i,\rho} = \oneop$ for all $\rho\in\DHR^{I}\{\A\}$. But also $\eps_{\rho,\overline\rho^2_i} = \oneop$ for all $\rho\in\DHR^{I}\{\A\}$, hence $[\overline\rho_i]$ is degenerate. By modularity of the category all coefficients $a_i = 0$ for $i = 1,\ldots,n$ and we are left with $a = a_0$ because $\overline R_0 = \oneop$ can be chosen without loss of generality. In particular, $a\in\A(I_1 \cup I_2)$.

We now relax the assumption (ii) about the category and allow $\DHR$ endomorphisms of dimension $d_\rho > 1$. As above we have
$$a = \rho(a) = \sum_i a_i\, \eps_{\overline\rho^2_i,\rho}\,\overline R_i$$
for all $\rho\in\DHR^{I}\{\A\}$ but now the coefficients have different localization properties and we need a more refined argument. Then rewrite
$$a = \sum_i a_i\, \rho_i^1(\eps_{\rho,\overline\rho_i^2}\, \eps_{\overline\rho^2_i,\rho})\overline R_i$$
and consider for all $\rho\in\DHR^{I}\{\A\}$ a conjugate endomorphism $\overline \rho$ again localizable in $I$ and operators $\overline R_\rho\in\Hom_{\DHR\{A(I)\}}(\id,\rho \overline{\rho})$ as before. The latter are $\overline R_\rho \in \A(I)$ and can be normalized such that $\overline R_\rho ^* \overline R_\rho = d_\rho\oneop$. Then we can write 
$$a = d_\rho^{-1} \overline R_\rho ^* \overline R_\rho a = d_\rho^{-1} \overline R_\rho ^* a \overline R_\rho = d_\rho^{-1} \sum_i a_i\, \overline R_\rho^*\,\rho_i^1(\eps_{\rho,\overline\rho_i^2}\, \eps_{\overline\rho^2_i,\rho})\overline R_i \overline R_\rho$$
by locality, and using $\rho_i^1\overline\rho_i^2(\overline R_\rho^*) = \overline R_\rho^*$ we have also
$$a = \rho(a) = d_\rho^{-1} \sum_i a_i\, \rho_i^1\overline\rho_i^2(\overline R_\rho^*)\,\rho_i^1(\eps_{\rho,\overline\rho_i^2}\, \eps_{\overline\rho^2_i,\rho})\overline R_i \overline R_\rho$$
where on the right hand side we have formed a \lq\lq killing-ring'',
after \cite[Sec. 3]{BEK99}, in order to exploit \emph{modularity}. Then choose one representative for each sector $\rho_j\in\DHR^{I}\{\A\}$ where $j=0,\ldots,n$ and consider
$$(\sum_j d_{\rho_j}^2)\, a = \sum_j d_{\rho_j}^2\, \rho_j(a) = \sum_{i,j} a_i\, d_{\rho_j} \rho_i^1\overline\rho_i^2(\overline R_{\rho_j}^*)\,\rho_i^1(\eps_{{\rho_j},\overline\rho_i^2}\, \eps_{\overline\rho^2_i,\rho_j})\overline R_i \overline R_{\rho_j}$$
$$ = \sum_i a_i\, (\sum_k d_{\rho_k}^2)\, \delta_{[\overline\rho_i],[\id]} \overline R_i = (\sum_k d_{\rho_k}^2)\, a_0 \overline R_0$$
by unitarity of the $S$-matrix, as shown by \cite{Reh90} in the case of UMTCs. As before we conclude $a = a_0 \in\A(I_1 \cup I_2)$.

It remains the case when $a\in\A\smallsetminus\Aloc$ relaxing assumption (i).
By the split property (assumption (b)) we have that $\A(I)$ is injective hence generated by an amenable group of unitaries. Averaging over its adjoint action (cf.\ proof of Proposition \ref{prop:relessduality}) we get a conditional expectation $E: \B(\Hilb) = \A(I) \vee \A(I)' \rightarrow \A(I)'$
mapping for all $I\Subset K$, $K\in\cI$
$$E(\A(K)) = \A(K) \cap \A(I)' , \quad E(\A) = \A(I)^c.$$
Since $E$ is norm continuous we have
$$\A(I)^c = \text{\Cstar}(\cup_{n\in\NN}\, \A(K_n) \cap \A(I)') ,\quad I\Subset \K_n\nearrow \RR\,,\; K_n\in\cI$$
hence we can write $a = \lim_{n} a_n$ where $a_n\in \A(K_n) \cap \A(I)'$. As in the previous steps we get
$$a_n = \sum_i a_{n,i} \overline R_i$$
where we can choose $\overline R_i$ independently of $n$ (at least for big $n$).
From the assumptions and norm continuity of $\rho\in\DHR^{I}\{\A\}$ we have
$$a = \rho(a) = \lim_{n} \rho(a_n) = \lim_n \sum_i a_{n,i}\, \eps_{\overline{\rho}_i^2,\rho}\, \overline R_i.$$
Now we show that for all $i$ the sequences $(a_{n,i})_n$ converge to some $b_i\in\A({I}')$. Indeed the coefficients are explicitly given \cite[Eq.\ (4.10)]{LoRe95} as 
$$ a_{n,i} = \lambda\, E_n(a_n \overline R_i^*)$$
where $\lambda$ is the $\mu_2$-index of the two-interval subfactor and we denoted by $E_{n} : \A({K_n})\cap\A(I)' \rightarrow \A(K_n \cap {I}')$ the minimal conditional expectations, see \cite[Prop.\ 5]{KLM01}. Compute 
$$\| a_{n,i} - a_{m,i} \| = \lambda\, \|E_n(a_n \overline R_i^*) - E_m(a_m \overline R_i^*)\|$$
but now it holds \cite[Lem.\ 11]{KLM01} that ${E_m}_{\restriction \A({K_n})\cap\A(I)'} = E_n$ if $m>n$, thus
$$ \lambda\, \|E_m((a_n - a_m) \overline R_i^*)\| \leq \lambda\, (d_{\rho_i})^{1/2} \|a_n - a_m\| \longrightarrow 0 $$
for $n,m \rightarrow \infty$. Then $(a_{n,i})_n$ are Cauchy
sequences. Since $\A({I}')$ is by definition norm closed, the limit points $b_i\in\A({I}')$ exist. Hence we have shown that the (local) unique expansion formula (\ref{eq:harmonicexpansion}) makes sense also in the quasilocal limit for the inclusion $\A({I}') \subset \A(I)^c$
\begin{equation}a = \sum_i b_i \overline R_i.\end{equation}
With the same argument as in the (local) two-interval case we can show that $\rho(a)=a$ for all $\rho\in\DHR^{I}\{\A\}$ implies $b_i = 0$ whenever $i\neq 0$, hence $a=b_0\in\A({I}')$ and the proof is complete.
\end{proof}

\begin{remark}\label{rmk:dop82conj}
A statement similar to the previous proposition appears in \cite{Dop82} as a \lq\lq natural conjecture'' which explains the shape of the inclusion $\A(\O') \subset \A(\O)^c$ where $\O$ is any open double cone region in Minkowski spacetime $\RR^{3+1}$. The generators of the extension can be interpreted in that case as local measurements of (global Abelian) superselection charges, see also \cite{DoLo83}. The situation here is much different: $\DHR$ superselection charges in low dimensions have non-degenerately braided statistics (opposite to permutation group), the category is modular instead of symmetric, there is no global gauge symmetry and the generators of the extension $\A(I')\subset \A(I)^c$, where $I\in\cI$, seem to have a purely topological nature. Surprisingly (in the light of the previous facts) the proof of the statement relies essentially on modularity. To our knowledge, by now there is no other proof of the statement in different contexts.
\end{remark}

From the previous proof, we also get the following

\begin{corollary}\label{cor:reallinetwointerval}
With the assumptions of Proposition \ref{prop:dualoflocalcat}, every element $a\in\A(I)^c = \A(I)'\cap\A$ admits a unique \lq\lq harmonic'' expansion, cf.\ \emph{\cite[Eq.\ (4.10)]{LoRe95}}
$$a = \sum_{i=0,\ldots,n} b_i \overline R_i$$
where $b_i\in\A({I}')$ are uniquely determined coefficients and $\overline R_i\in\Hom(\id,\rho^1_i \overline{\rho}^2_i)\subset\A(I)^c$ are (fixed) generators of the extension of unital \Cstar-algebras
$$\A({I}')\subset \A(I)^c.$$

In particular, for holomorphic conformal nets it holds (cf.\ Proposition \ref{prop:relessduality})
$$\Aholo({I}') = \Aholo(I)^c.$$
\end{corollary}

\begin{remark}\label{rmk:dualofwedgecat}
Relations analogous to Proposition \ref{prop:dualoflocalcat} hold for half-lines $W\subset\RR$, namely ${\DHR^W\{\A\}}^\perp = \A(W')$ as one can easily show using Proposition \ref{prop:relessduality}. We shall see later a more general argument, see Proposition \ref{prop:4tuple}.
\end{remark}

\section{Local duality relations}\label{sec:localdualityrelations}

We turn now to the local picture, i.e., consider as environment some local algebra $\A(I_0)$ for arbitrarily fixed $I_0\in\cI$ instead of the quasilocal algebra $\A$. Similarly to (\ref{eq:dualpar}) we consider the \emph{local} duality pairing
\begin{equation}\label{eq:localdualpar}\A(I_0)\,\stackrel{\perp}{\longleftrightarrow}\,\DHR^{I_0}\{\A\}.\end{equation}
The local version of all the statements we made in Section \ref{sec:dualityrelations} follows analogously, thanks to
strong additivity, by considering \emph{local} interval algebras
$\A(I)\subset \A(I_0)$ if $I\Subset I_0$, $I\in\cI$, and \emph{local}
half-line algebras $\A(I_1)\subset\A(I_0)$ if $I_1=W\cap I_0$,
$W\subset\RR$ is any half-line with origin $p\in I_0$. 

In the following the symbol $^\perp$ will refer to (\ref{eq:localdualpar}). Similarly to the notion of relative commutant for unital inclusions of algebras, i.e., $\N^c = \N'\cap \A(I_0)$ if $\N\subset\A(I_0)$, we introduce relative commutants of subcategories

\begin{definition}\label{def:relcommsubcat}
Let $\C\subset \DHR^{I_0}\{\A\}$ be a unital full inclusion of tensor categories, we define the \textbf{relative commutant} as 
$$\C^c := \big\{\rho\in\DHR^{I_0}\{\A\} : \rho\,\sigma = \sigma\rho,\,\sigma\in\C\big\}$$ 
where the equality sign means pointwise equality as endomorphisms of $\A(I_0)$, or equivalently of $\A$. We define $\C^c\subset \DHR^{I_0}\{\A\}$ as a full subcategory, i.e., $\Hom_{\C^c}(\rho,\sigma) := \Hom_{\DHR\{\A\}}(\rho,\sigma)$ for every $\rho,\sigma\in\C^c$.
\end{definition}

$\C^c$ is automatically a unital tensor category of endomorphisms of $\A(I_0)$. Now combining relative commutants and duals, given a subalgebra $\N\subset\A(I_0)$ we define a unital tensor full subcategory $\C_\N \subset \DHR^{I_0}\{\A\}$ as
$$\C_\N := {\N^c}^\perp$$
where by definition $\Hom_{\C_\N}(\rho,\sigma) = \Hom_{\DHR\{\A\}}(\rho,\sigma)$ for every $\rho,\sigma\in\C_\N$.

\begin{remark}\label{rmk:C_Nareglobal}
Despite we use the term \lq\lq local"  for the duality pairing
(\ref{eq:localdualpar}) and for the respective subcategories of
$\DHR^{I_0}\{\A\}$ defined as above, it should be kept in mind that both
$\C_\N$ and $\DHR^{I_0}\{\A\}$ are categories of globally defined endomorphisms
of the quasilocal algebras $\A$, which then are \lq\lq localizable" in
smaller regions, e.g., $I_0$, i.e., act trivially on every local algebra
$\A(J)$, $J\subset I_0'$ and on $\N^c$.
\end{remark}

Summarizing the previous results, we have 
\begin{corollary}\label{cor:honestpoints}
Let $p\in I_0$ and $I_0\smallsetminus\{p\}=I_1\cup I_2$. Let $\N:=\A(I_1)$,
then $\N^c=\A(I_2)$, $\C_\N=\DHR^{I_1}\{\A\}$, $\C_{\N^c}=\DHR^{I_2}\{\A\}$. 
Moreover, if $I_1$ is to the left of $I_2$, then $\eps_{\rho,\sigma} = \oneop$ whenever $\rho\in\C_\N$, $\sigma\in\C_{\N^c}$.
\end{corollary}

\begin{remark}\label{rem:points} It is well known that a
point as the localization of an observable is an over-idealization,
forcing fields to be distributions, and making the intersections of
local algebras corresponding to regions intersecting at a point
trivial. In contrast, the proper way of ``lifting'' points to
quantum field theory rather seems to be their role as separators
between local algebras, trivializing the braiding as in Corollary \ref{cor:honestpoints}.
\end{remark}

\section{Abstract points}\label{sec:apts}

Let $\{\A\}$ be a completely rational conformal net on the line (Definition \ref{def:RCFTline}). In the previous two sections we essentially used the action of the $\DHR$ category, and its abstract structure of UMTC. Now we employ the $\DHR$ braiding as well, see equation (\ref{eq:QFTbraiding}) and comments thereafter, hence the braided action (Definition \ref{def:braction}) given by the \emph{restriction functor}
$$\C:=\DHR^{I_0}\{\A\}\hookrightarrow\End(\M_0)$$
where $\M_0 := \A(I_0)$ and $I_0\in\cI$ is an arbitrarily fixed interval. 

\begin{definition}\label{def:apts}
We call \textbf{abstract point of $\M_0$} an ordered pair of algebras $(\N,\N^c)$ where $\N\subset\M_0$ such that
\begin{itemize} \itemsep0mm
\item [(i)] $\N$ and $\N^c$ are injective type $\III_1$ factors.
\item [(ii)]$\N = \N^{cc}$ and $\N\vee\N^c = \M_0$.
\item [(iii)] $\C_\N \simeq \C$ and $\C_{\N^c} \simeq \C$ as UBTCs.
\item [(iv)] $\eps_{\rho,\sigma} = \oneop$ whenever $\rho\in\C_\N$, $\sigma\in\C_{\N^c}$.
\end{itemize}
With abuse of notation we denote abstract points by $p := (\N,\N^c)$, and call $\N$, $\N^c$ respectively the \textbf{left}, \textbf{right relative complement of $p$ in $\M_0$}.
\end{definition}

More generally, given an \lq\lq abstract'' UMTC $\C$ together with a braided action on the injective type $\III_1$ factor $\M$, see Definition \ref{def:braction} and Remark \ref{rmk:yamembedd}, we can analogously define abstract points of $\M$ (with respect to the braided action $\C\hookrightarrow\End(\M)$).
In the case of a UMTC coming from a completely rational conformal net, $\C = \DHR^{I_0}\{\A\}$ together with its \emph{canonical} braided action on $\M_0$, the existence of those is the content of the previous sections.

\begin{remark}
Condition (iii) is indeed equivalent to essential surjectivity of the inclusion functors $\C_\N \subset \C$ and $\C_{\N^c} \subset \C$. In fact $\C_\N \subset \C \subset \DHR\{\A\}$ are full inclusions by definition, the latter also essentially surjective, and the inclusion functor is trivially unitary strict tensor and braided.
\end{remark}

\begin{remark}\label{rmk:finmanyeqs}
Condition (iv) consists a priori of uncountably many constraints on braiding operators. We shall see in Proposition \ref{prop:naturalityiefinmanyeqs} that it is indeed equivalent to a finite system of equations. This makes (iv) a more tractable (\lq\lq rational") condition.
\end{remark}

\begin{remark}\label{rmk:hpts}
From Corollary \ref{cor:honestpoints} we know that ordered pairs of local
algebras $(\A(I_1),\A(I_2))$, associated respectively to the left
and right relative complements $I_1, I_2$ of some $p\in I_0$, are
also abstract points of $\M_0 = \A(I_0)$. We shall refer to them as
\emph{honest} points of $\M_0$ (with respect to the net
$\{\A\}$). The converse is not true in general, see in Sections \ref{sec:fuzzyapts} and \ref{sec:primeCFTline}.
\end{remark}

At the level of generality of Definition \ref{def:apts} we can show the following

\begin{proposition}\label{prop:4tuple}
Let $p=(\N,\N^c)$ be an abstract point of $\M_0$, then the quadruple $(\N,\N^c,\C_\N,\C_{\N^c})$ is uniquely determined by any one of its elements.
\end{proposition}

\begin{proof}
It is sufficient to show that $\C_{\N^c}$ determines $\N$. By definition ${\C_{\N^c}}^\perp = {\N^{cc}}^{\perp\perp} = {\N}^{\perp\perp}$ holds and the inclusion $\N\subset{\N}^{\perp\perp}$ is trivial. The opposite inclusion also holds for algebras of the form $\N = \cP^c$, where $\cP\subset\M_0$ is any unital \Cstar-subalgebra of $\M_0$, cf.\ \cite[Sec. 5]{Dop82}, in our case $\cP = \N^c$. Let $a\in\N^{\perp\perp}$ and consider the unitary group $\cU(\cP)$, then $\Ad_u\in\N^\perp$ for all $u\in\cU(\cP)$ hence $\Ad_u(a)=a$ and we conclude $a\in\cU(\cP)'$. Now $\cU(\cP)$ linearly spans $\cP$, hence $a\in\M_0\cap\cP' = \cP^c = \N$.
\end{proof}

The gain in considering together pairs of subfactors or pairs of subcategories is that we can use the braiding operators between endomorphisms as a remnant of their localization properties (left/right separation) hence, dually, of the net. The first interesting consequence of Definition \ref{def:apts} is however the following

\begin{proposition}\label{prop:abstractresfun}
Let $(\N,\N^c)$ be a pair of subfactors of $\M_0$ fulfilling
conditions (i) and (ii) in the Definition \ref{def:apts} of abstract
points.
 
If we consider for instance $\N\subset\M_0$ and the associated $\C_\N\subset \C$, we have
\begin{itemize} \itemsep0mm
\item if $\rho\in\C_\N$ then $\rho\in\End(\N)$.
\item if $t\in\Hom_{\C_\N}(\rho,\sigma)$ where $\rho,\sigma\in\C_\N$, then $t\in\N$.
\item if $t\in\N$ and $t \rho(n) = \sigma(n) t$ for all $n\in\N$ where $\rho,\sigma\in\C_\N$, then $t\in\Hom_{\C_\N}(\rho,\sigma)$.
\end{itemize}
In other words, we have a well-defined, faithful and full restriction functor $\rho\mapsto\rho_{\restriction\N}$
$$\C_\N \hookrightarrow \End(\N).$$
\begin{itemize} \itemsep0mm
\item if $\rho\in\C_\N$ and $u\in\cU(\N)$ then $\Ad_u\rho\in\C_\N$.
\end{itemize}
Hence the restriction functor has replete image, i.e., it is
specified by its sectors (unitary isomorphism classes of objects) only.
\end{proposition}

\begin{proof}
First, take $\rho\in\C_{\N} = {\N^c}^\perp$ and $n\in\N$, then $\rho(n)m=\rho(nm)=m\rho(n)$ for all $m\in\N^c$ and we get $\rho(n)\in\M_0\cap\N^{c \,\prime} = \N^{cc} = \N$.

Second, take $t\in\M_0$ such that $t\rho(a) = \sigma(a)t$ for all $a\in\M_0$, where $\rho,\sigma\in\C_\N$. Now, letting $a\in\N^c$ we have $ta = at$ hence $t=\N^{cc} = \N$.

Third, we have $t\in\N$ and $t \rho(n) = \sigma(n) t$ if $n\in\N$ by definition and $t \rho(m) = \sigma(m) t$ if $m\in\N^c$ because $tm=mt$. Now, every $a\in\M_0=\N\vee\N^c$ can be written as an ultra-weak limit of finite sums $a=\text{\emph{uw-}}\lim \sum_i n_i m_i$ where $n_i\in\N$ and $m_i\in\N^c$. Also, $\rho,\sigma$ are automatically normal on $\M_0$, see \cite[p. 352]{Tak1}, being $\M_0$ non-type $I$ and $\Hilb$ separable. 
Normality on $\M_0 = \A(I_0)$ can also be derived by $\DHR$ transportability of the endomorphisms, but we prefer the previous argument which is intrinsic and local. From these two facts we conclude that $t\rho(a) = \sigma(a) t$ for all $a\in\M_0$, hence as $\DHR$ endomorphisms because local intertwiners are global, i.e., $\C\hookrightarrow\End(\M_0)$ is full.

The last point is trivial to show, but has interesting consequences
(see Proposition \ref{prop:rigidityfusion}).
\end{proof}

The conditions stated in Definition \ref{def:apts} contain many redundancies. Out of the operator algebraic assumptions (i) and (ii) on $\N$ and $\N^c$, one can derive properties of their dual categories $\C_\N$ and $\C_{\N^c}$ which are custom assumptions in \Cstar tensor category theory, see, e.g., \cite{LoRo97}. Nevertheless, assumptions (iii) and (iv) cannot be derived from the previous, see Proposition \ref{prop:dualofwedge} and \ref{prop:dualofsergio}, unless the net $\{\A\}$ is holomorphic.

\begin{proposition}\label{prop:rigidityfusion}
Let $(\N,\N^c)$ be a pair of subfactors of $\M_0$ fulfilling
conditions (i) and (ii) in the Definition \ref{def:apts} of abstract
points. Then the subcategories $\C_\N$ and $\C_{\N^c}$ automatically
have irreducible tensor unit, subobjects, finite direct sums and
conjugate objects.
 
In other words, they are \Cstar tensor categories which are also fusion and rigid.
\end{proposition}

\begin{proof}
The restriction functor $\C_\N \hookrightarrow \End(\N)$ is full and
faithful by Proposition \ref{prop:abstractresfun}, hence
irreducibility of the tensor unit of $\C_\N$ is equivalent to
factoriality of $\N$. 

In general the existence of subobjects in $\DHR\{\A\}$ follows because
we have a \emph{net} of type $\III$ factors, i.e., $\A(I_0)$ alone
being type $\III$ is not sufficient to construct $\DHR$
subendomorphisms. In our case we need again Proposition
\ref{prop:abstractresfun} together with $\N$ being type $\III$. Let
$\rho\in\C_\N$ and $e\in \Hom_{\C_\N}(\rho,\rho)\subset\N$ a
non-zero orthogonal projection. Choose $v\in\N$ such that
$v^*v = \oneop$, $vv^* = e$ and let $\sigma(n) := v^* \rho(n) v$,
$n\in\N$, then $\sigma\in\End(\N)$ by definition. In order to show
$\sigma\prec\rho$ in $\C_\N$ we need to extend $\sigma$ to $\M_0$ and
then to the quasilocal algebra $\A$, in such a way that the intertwining
relation $v\in\Hom_{\C_\N}(\sigma,\rho)$ holds, cf.\ Remark \ref{rmk:C_Nareglobal}. Now $\sigma(m) :=
v^*\rho(m)v = m$, $m\in\N^c$, and $\rho$ is normal on $\M_0$ hence
$\sigma$ extends to $\End(\M_0)$ with $\sigma_{\restriction \N^c} =
\id$ and $v\in\Hom_{\End(\M_0)}(\sigma,\rho)$. On the other hand
$\rho\in\C$ and $\C$ has subobjects, hence let $w\in\M_0$ and
$\tau\in\C$ such that $w^*w = \oneop$, $ww^* = e$ and
$w\in\Hom_{\C}(\tau,\rho) = \Hom_{\End(\M_0)}(\tau,\rho)$. Now $w^*v $
is unitary in $\Hom_{\End(\M_0)}(\sigma,\tau)$ hence we can extend
$\sigma\in\C$ because $\C\hookrightarrow\End(\M_0)$ is replete. Thus
$\sigma\in\C_\N$ and $v\in\Hom_{\C_\N}(\sigma,\rho)$ because
$\C_\N\hookrightarrow\End(\N)$ is full. 

Along similar lines one can show the existence of direct sums in $\C_\N$.

To show existence of conjugates in $\C_\N$ we need, in addition,
results from the theory of infinite subfactors with finite index. Let
$\rho\in\C_\N$ be an irreducible $\DHR$ endomorphism, hence with
finite (minimal) index $\Ind(\rho(\M_0),\M_0) < \infty$ \cite[Cor.\
39]{KLM01}, i.e., finite statistical dimension $d_\rho < \infty$
\cite[Cor.\ 3.7]{GuLo96}. Let $\Phi$ be the unique left inverse of
$\rho$, see \cite[Cor.\ 2.12]{GuLo96}, which is normal on $\M_0$ and
localizable in $I_0$, hence in particular $\Phi(\M_0) \subset
\M_0$. For every $n\in\N$, $m\in\N^c$ we have $\Phi(m) = \Phi(\rho(m))
= m$ and $\Phi(n)m = \Phi(n\rho(m)) = \Phi(nm) = m\Phi(n)$ hence
$\Phi_{\restriction\N^c} = \id$ and $\Phi(\N)\subset\N^{cc} = \N$.
 
Again by Proposition \ref{prop:abstractresfun}, irreducibility of $\rho$ is equivalent to irreducibility of the subfactor $\rho(\N)\subset\N$, then $E_{\restriction\N}:=\rho\circ\Phi_{\restriction\N}$ coincides with the unique normal faithful (minimal) conditional expectation given by \cite[Thm.\ 5.5]{Lon89}.
After setting $\lambda := \Ind(\rho(\M_0),\M_0)^{-1}$, we have the Pimnser-Popa bound \cite[Thm.\ 4.1]{Lon89} 
\begin{equation}\label{eq:PiPobound} E(a^*a) \geq \lambda a^*a, \quad a\in\M_0  \end{equation}
where $\lambda$ is the best constant fulfilling equation (\ref{eq:PiPobound}).  
In particular, it holds for all $a\in\N\subset\M_0$ and if we let $\mu
:= \Ind(\rho(\N),\N)^{-1}$ by the same argument on $\rho(\N)\subset\N$
and by uniqueness of $E_{\restriction\N}$
we get $\mu\geq\lambda$, hence $\Ind(\rho(\N),\N) < \infty$. 

Now we turn to the construction of the conjugate endomorphism of
$\rho$ in $\C_\N$. As before we begin \lq\lq locally'', i.e., by
construction of the restriction of the conjugate as an object of $\End(\N)$, and then extend. Let $\rho_\N := \rho_{\restriction\N}\in\End(\N)$ and $\rhobar := (\rho_\N)^{-1} \circ \gamma\in\End(\N)$ where $\gamma$ is a canonical endomorphisms of $\N$ into $\rho(\N)$ \cite[Thm.\ 3.1]{Lon90}. By finiteness of the index of $\rho(\N)\subset\N$ \cite[Thm.\ 4.1 and 5.2]{Lon90} we have a solution $R\in\Hom_{\End(\N)}(\id,\rhobar\rho_\N)$, $\overline{R}\in\Hom_{\End(\N)}(\id,\rho_\N\rhobar)$ of the conjugate equations \cite[Sec. 2]{LoRo97} in $\End(\N)$. First, we extend $\rhobar$ to $\M_0$ by making use of another formula for the canonical endomorphism \cite[Eq.\ (2.19)]{LoRe95}
\begin{equation}\label{eq:anothercanendo} \gamma(n) = \lambda d_\rho^{-1}  E(\overline{R} n \overline{R}^*),\quad n\in\N. \end{equation}
By (\ref{eq:anothercanendo}) $\gamma$ extends normally to $\M_0$ and
to the quasilocal algebra $\A$. Also, for $m\in\N^c$ we get $\gamma(m) =
\lambda d_\rho^{-1}  E(\overline{R} m \overline{R}^*) = \lambda
d_\rho^{-1}  E(\overline{R} \overline{R}^*) m = m$ by
\cite[Eq.\ (4.1)]{LoRe95}, hence $\gamma_{\restriction \N^c} = \id$ and
$\gamma(\M_0) \subset \rho(\M_0)$. It follows that we can extend
normally $\rhobar := \rho^{-1}\circ\gamma\in\End(\M_0)$ because $\rho$ is
injective hence bicontinuous onto its image in the ultraweak topology
\cite[p. 59]{Ped79}. Moreover we have $\rhobar_{\restriction \N^c} =
\id$ and $R\in\Hom_{\End(\M_0)}(\id,\rhobar\rho)$,
$\overline{R}\in\Hom_{\End(\M_0)}(\id,\rho\rhobar)$.

On the other hand $\rho\in\C$ and let $\tilde\rho\in\C$ be a $\DHR$
conjugate of $\rho$, hence by irreducibility and \cite[Thm.\
3.1]{Lon90} we have a unitary
$u\in\Hom_{\End(\M_0)}(\rhobar,\tilde\rho)$. As above we extend
$\rhobar\in\C$ by repleteness of $\C\hookrightarrow\End(\M_0)$, hence
$\rhobar\in\C_\N$ together with $R\in\Hom_{\C_\N}(\id,\rhobar\rho)$,
$\overline{R}\in\Hom_{\C_\N}(\id,\rho\rhobar)$, and we have the
statement in the irreducible case.

Now $R, \overline R$ can be normalized in such a way $R^*R = \overline
R^* \overline R$ gives the (intrinsic) dimension of $\rho$ in
$\C_{\N}$. The latter does not depend on the choice of normalized
solutions in $\C$, and equals the statistical dimension $d_\rho$ on
one side and $\Ind(\rho(\N),\N)^{1/2}$ on the other by
\cite[p. 121]{LoRo97}. In particular, it holds $\lambda = \mu$ and
${d_\rho}^2 = \Ind(\rho(\N),\N)$. 
 
The construction of conjugates extends to finite direct sums, concluding the proof of the proposition for $\C_\N$. Similarly for $\C_{\N^c}$ interchanging the roles of $\N$ and $\N^c$.
\end{proof}

\begin{remark}
See \cite[Thm.\ 2.2, Cor.\ 2.4]{GuLo92} for a similar discussion on the
conjugation of endomorphisms of \emph{sub}factors.
\end{remark}

Going back to the duality between subalgebras and subcategories, under assumption (iii) we can lift the normality relations contained in (ii) from $\N,\N^c$ to $\C_\N, \C_{\N^c}$, in the sense of Definition \ref{def:relcommsubcat}.

\begin{proposition}\label{prop:cd}
Let $(\N,\N^c)$ be a pair of subfactors of $\M_0$ fulfilling conditions (i), (ii) and (iii) in the Definition \ref{def:apts} of abstract points. Then 
$$(\C_\N)^c = \C_{\N^c}, \quad (\C_{\N^c})^c = \C_{\N}$$
and the operations in the diagram
\[
\begin{array}{ccc}
  \N & \stackrel{\perp}{\longmapsto} & \C_{\N^c} \\ 
\scriptstyle{c}\longdownmapsto && \longdownmapsto\scriptstyle{c} \\ 
  \N^c & \stackrel{\perp}\longmapsto  &  \C_\N
\end{array}
\]
are commutative and invertible.
\end{proposition}

\begin{proof}
Take $\rho\in\C_{\N^c}$ and first assume (iv) in addition, then $\eps_{\sigma,\rho} = \oneop$ for all $\sigma\in\C_\N$ gives in particular $\rho\,\sigma = \sigma\rho$ and we can conclude $\rho\in(\C_\N)^c$. But we want the statement independent of braiding operators, hence we use Proposition \ref{prop:abstractresfun} to draw the same conclusion. Indeed $\rho(\sigma(m)) = \rho(m) = \sigma(\rho(m))$ for all $\sigma\in\C_{\N}$ and $m\in\N^c$, and the same holds for $n\in\N$. As before, by assumption (i) and (ii) we have $\M_0=\N\vee\N^c$ and $\rho,\sigma$ are normal on $\M_0$.
Hence $\rho\,\sigma = \sigma\rho$ for all $\sigma\in\C_{\N}$ and again $\rho\in(\C_\N)^c$.

Viceversa, if $\rho\in(\C_\N)^c$ then in particular $\rho\Ad_u = \Ad_u \rho$ for all $u\in\cU(\N)$, explicitly $\rho(uau^*) = u\rho(a)u^*$ for all $a\in\M_0$.
Then we have $u^*\rho(u) \in \Hom_{\End(\M_0)}(\rho,\rho) = \Hom_{\C}(\rho,\rho)$. If $\rho$ is irreducible, then $u^*\rho(u) = \lambda_u$ where $\lambda_u\in\mathbb{T}$ is a complex phase. 
The map $u\in\cU(\N)\mapsto \lambda_u\in\mathbb{T}$ is a norm
continuous unitary character, hence trivial by \cite[Thm.\ 1]{Kad52}
because $\N$ is a non-type $I$ factor by assumption (i), and we have
$\rho(u) = u$ for all $u\in\cU(\N)$. In this case, we conclude
$\rho\in\N^\perp = \C_{\N^c}$.

In general, if $\rho\in(\C_\N)^c$ is (finitely) reducible, we can
write $\rho$ as a finite direct sum of irreducibles $\rho =
\oplus_{i=1,\ldots,n} \rho_i$ with $\rho_i\in\C_{\N^c}$ by assumption
(iii). Notice that we already have the inclusion $\C_{\N^c} \subset
(\C_\N)^c$. Let $\rho,\sigma\in(\C_\N)^c$ and
$t\in\Hom_{(\C_\N)^c}(\rho,\sigma)$, then one has
$$\Ad_{u}(t) \rho(\Ad_u(a)) = \sigma(\Ad_u(a)) \Ad_u(t)$$
for every $u\in\cU(\N)$, because $\Ad_u\in\C_\N$. But every $\Ad_u$ is an automorphisms of $\M_0$ hence we get $\Ad_u(t)\in\Hom_{(\C_\N)^c}(\rho,\sigma)$ and $u\in\cU(\N) \mapsto \Ad_u$ is a group representation of $\cU(\N)$ on the finite-dimensional vector space $V := \Hom_{(\C_\N)^c}(\rho,\sigma)$, see \cite[Lem.\ 3.2]{LoRo97}.
Now, $V^*V = \Hom_{(\C_\N)^c}(\rho,\rho)$ is isomorphic to a finite-dimensional block-diagonal matrix algebra, e.g., if $n=2$ then $\Hom_{(\C_\N)^c}(\rho_1\oplus\rho_2,\rho_1\oplus\rho_2)$ is either the full matrix algebra $M_2(\CC)\cong \CC^4$ if $\rho_1 \cong \rho_2$ or diagonal matrices $\Lambda_2(\CC) \cong \CC^2$ if $\rho_1 \ncong \rho_2$. Hence we can consider the Hilbert inner product on $V$ given by the (non-normalized) trace of $V^*V$, i.e.
$$(t|s) := \Tr(t^*s) = \sum_{i=1,\ldots,n} t_i^* (t^*s) t_i$$
where $t,s\in V$ and $\{t_1,\ldots,t_n\}\subset\M_0$ is a Cuntz
algebra of isometries defining $\rho = \oplus_i \rho_i$, namely
$t_i^*t_j = \delta_{i,j}$ and $\sum_i t_it_i^* = \oneop$ and
$t_i\in\Hom_{(\C_\N)^c}(\rho_i,\rho)$. The definition of trace does not depend on the choice of $\{t_1,\ldots,t_n\}$ and that matrix units of $V^*V$ form an orthonormal basis of $V^*V$ with respect to the previous inner product. Now, given $t,s\in V$ and $u\in\cU(\N)$ compute
$$(\Ad_u(t)|Ad_u(s)) = \Tr(ut^*su^*) =  \Tr(\rho(u)\rho(u^*)ut^*su^*\rho(u)\rho(u^*))$$
$$= \sum_{i=1,\ldots,n} t_i^* (\rho(u)\rho(u^*)ut^*su^*\rho(u)\rho(u^*)) t_i = u \Tr(\rho(u^*)ut^*su^*\rho(u)) u^*$$
$$= \Tr(t^*s) = (t|s)$$
because $\rho_i(u) = u$, being $\rho_i\in\C_{\N^c}$, and $u^*\rho(u)\in V^*V$ so we can use the trace property. 
Hence the representation of $\cU(\N)$ on $V$ is unitary with respect
to the previous inner product, and norm continuous, as one can easily
check with respect to the induced \Cstar-norm of $V\subset\M_0$ and
then using the equivalence of norms for finite-dimensional vector
spaces. Again by \cite{Kad52} and assumption (i) the representation
must be trivial, i.e., $\Ad_u(t) = t$ for all $u\in\cU(\N)$, hence
$t\in\N'\cap\M_0 = \N^c$ and we have shown
$\Hom_{(\C_\N)^c}(\rho,\sigma) \subset \N^c$.
 
In conclusion, we get that every Cuntz algebra of isometries defining the direct sum $\rho = \oplus_i \rho_i$ lies in $\N^c$, hence we conclude $\rho\in\C_{\N^c}$. Both subcategories $\C_{\N^c}$ and $(\C_\N)^c$ are full by definition, hence they have the same Hom-spaces, and the proof is complete.
\end{proof}

Concerning condition (iv) in Definition \ref{def:apts}, the following shows that the braiding contains \emph{all} the information about the subcategories $\C_\N$, $\C_{\N^c}$ and charge transportation among them.

\begin{lemma}\label{lem:abstrchargetranspo}
Let $p = (\N,\N^c)$ be an abstract point of $\M_0$. Let $\rho\in\C$, then
\begin{itemize} 
\item $\rho\in\C_\N$ if and only if $\eps_{\rho,\Ad_u}=\oneop$ for all $u\in\cU(\N^c)$. 
\end{itemize}
Let $\rho\in\C$, $v\in\cU(\M_0)$ and set $\tilde\rho := \Ad_v
\rho$. We call $v$ an \textbf{abstract $\rho$-charge transporter to
  $\C_{\N^c}$} if it holds $\sigma(v) = v
\eps_{\sigma,\rho}\;\;\text{for all } \sigma\in\C_\N$. The terminology
is motivated by the following equivalence 
\begin{itemize}
\item $\tilde\rho\in\C_{\N^c}$ if and only if $v$ is an abstract $\rho$-charge transporter to $\C_{\N^c}$. 
\end{itemize}
Analogous statements hold interchanging $\N$ with $\N^c$ and $\eps$
with $\eps^\op$. \footnote{The \emph{opposite} braiding of $\C$ is defined as $\eps^\op_{\rho,\sigma} := \eps^*_{\sigma,\rho}$, or equivalently by interchanging left and right localization in the $\DHR$ setting.}
\end{lemma}

\begin{proof}
By naturality of the braiding and using the convention
$\eps_{\rho,\id} = \oneop$ we see that triviality of braiding
operators with inner automorphisms $\Ad_u$ is triviality of
the action of the endomorphism on $u$. Hence the first statement follows.

For the second, take $\rho\in\C$ and $v\in\cU(\M_0)$ an abstract $\rho$-charge transporter to $\C_{\N^c}$. For every $\sigma\in\C_\N$, $a\in\M_0$ compute $\sigma \tilde\rho (a) = \sigma(v) \sigma\rho(a) \sigma(v^*) = v\eps_{\sigma,\rho} \sigma\rho(a) \eps_{\sigma,\rho}^* v^* = \tilde\rho\, \sigma(a)$ hence $\tilde\rho\in(\C_\N)^c = \C_{\N^c}$ by Proposition \ref{prop:cd}. Viceversa, if $\tilde\rho = \Ad_v\rho \in\C_{\N^c}$ for some $v\in\cU(\M_0)$ then $\eps_{\sigma,\tilde\rho}=\oneop$ for every $\sigma\in\C_\N$ by (iv). Hence $v \eps_{\sigma,\rho} \sigma(v^*) = \oneop$ and we obtain the second statement.
\end{proof}

On the other hand, after defining $\C_\N$, $\C_{\N^c}$ by duality from $\N$, $\N^c$, condition (iv) turns out to be equivalent to a finite system of equations.

\begin{proposition}\label{prop:naturalityiefinmanyeqs}
Let $(\N,\N^c)$ be a pair of subfactors of $\M_0$ fulfilling conditions (i), (ii) and (iii) in the Definition \ref{def:apts} of abstract points. For each sector labelled by $i=0,\ldots,n$ choose (assumption (iii)) irreducible representatives $\rho_i\in\C_\N$ and $\sigma_i\in\C_{\N^c}$ respectively in $\C_\N$ and $\C_{\N^c}$, such that $[\rho_i] = [\sigma_i]$. Then
$$\eps_{\rho_i,\sigma_j} = \oneop,\quad i,j=0,\ldots,n$$
is equivalent to condition (iv).
\end{proposition}

\begin{proof}
In order to show the nontrivial implication, we first take $\rho\in\C_\N$ and $\sigma\in\C_{\N^c}$ irreducible. By Proposition \ref{prop:abstractresfun} we have $\Ad_{u_i} \rho = \rho_i$ and $\Ad_{v_j} \sigma = \sigma_j$ for some $i,j \in \{0,\ldots,n\}$ and $u_i\in\cU(\N)$, $v_j\in\cU(\N^c)$. Naturality of the braiding gives
$$\eps_{\rho,\sigma} = \sigma(u_i^*)v_j^*\eps_{\rho_i,\sigma_j}u_i\rho(v_j)$$
hence
$\eps_{\rho,\sigma} = \sigma(u_i^*)v_j^*u_i\rho(v_j) = \oneop$
because, e.g., $u_i\rho(v_j) = u_i v_j = v_j u_i$. Hence we have shown
(iv) in the irreducible case.

In the reducible case, we can write direct sums $\rho=\sum_a s_a\rho_a s_a^*$ and $\sigma=\sum_b t_b\sigma_b t_b^*$ where $a,b\in\{0,\ldots,n\}$ and $\rho_a\in\C_\N$, $\sigma_b\in\C_{\N^c}$ run in our choice of representatives and $\{s_a\}_a$, $\{t_b\}_b$ are Cuntz algebras of isometries respectively in $\N$, $\N^c$, again by Proposition \ref{prop:abstractresfun}. As before
$$\eps_{\rho,\sigma} = \sum_{a,b} \sigma(s_a) t_b \eps_{\rho_a,\sigma_b}s_a^*\rho(t_b^*) = \sum_{a,b} s_a s_a^* t_b t_b^* = \oneop$$
so we conclude (iv) for all $\rho\in\C_\N$, $\sigma\in\C_{\N^c}$.
\end{proof}

\begin{remark}
Thinking in terms of $\DHR$ localization properties of the endomorphisms, if we have $\rho\in\C_\N$, $[\rho]\neq[\id]$, the previous statement says that it cannot be localizable in some interval $I_\rho$ which is to the right of some localization intervals $I_{j}$ of $\sigma_j\in\C_{\N^c}$ as above, for all $j=0,\ldots,n$, for every choice of such $\sigma_j\in\C_{\N^c}$. This would imply degeneracy of $[\rho]$, hence contradict modularity of $\DHR\{\A\}$. Despite this naive left/right separation picture, and the results of the last section, we shall see next how abstract points can become wildly non-geometric or \lq\lq fuzzy". This is a typical situation in QFT where points of spacetime are replaced by (field) operators.  
\end{remark}

\section{Fuzzy abstract points}\label{sec:fuzzyapts}

Let $\{\A\}$ be a completely rational conformal net on the line, let
$I_0\in \cI$, $\M_0 = \A(I_0)$ and $\C = \DHR^{I_0}\{\A\}$. Inside
$\M_0$ we can find honest points (those associated to geometric points
$p\in I_0$, see Remark \ref{rmk:hpts}), but also uncountably many
families of abstract points which are \emph{fuzzy}, in the sense that
they are not honest anymore (with respect to $\{\A\}$) and do not
resemble any kind of geometric interpretation. The following examples
give algebraic deformations of abstract points into abstract points, and of honest
points into possibly fuzzy ones.

\begin{example}\label{ex:fatpts}
Let $p=(\A(I_1),\A(I_2))$ be an honest point of $\M_0$ and consider localizable unitaries $u\in\cU(\M_0)$. Then $upu^* := (\Ad_u(\A(I_1)),\Ad_u(\A(I_2)))$ is an abstract point of $\M_0$, see Definition \ref{def:apts}. Indeed conditions (i) and (ii) follow because $\Ad_u:\M_0\rightarrow \M_0$ is a normal automorphism, in particular $\Ad_u(\A(I_1)^c) = \Ad_u(\A(I_1))^c$. Now if $\rho\in\C_{\A(I_1)}$ then $^u\rho := \Ad_u \circ\,\rho\circ \Ad_{u^*}$ is again in $\C$ because $\Ad_u \circ\,\rho\circ \Ad_{u^*} = u\rho(u^*)\rho(\cdot)\rho(u)u^*$ and $u\rho(u^*)\in\cU(\M_0)$. Moreover it acts trivially on $\Ad_u(\A(I_1))^c$ hence $\rho\mapsto {^u\rho}$ defines a bijection between the objects of $\C_{\A(I_1)}$ and $\C_{\Ad_u(\A(I_1))}$, and (iii) follows. One easily checks that $\rho\mapsto{^u}\rho$ respects the tensor structure of $\C$, where the action on arrows $s\in\Hom_\C(\rho,\sigma)$, $\rho,\sigma\in\C$ is given by $^u s := \Ad_u(s)$. Condition (iv) is also fulfilled because $\rho\mapsto{^u}\rho$ respects the braiding of $\C$, namely
$$\eps_{^u\rho,^u\sigma} = u\sigma(u^*)\sigma(u\rho(u^*))\eps_{\rho,\sigma}\rho(\sigma(u)u^*)\rho(u)u^* = {^u}{\eps_{\rho,\sigma}}$$
by naturality, hence $\eps_{\rho,\sigma} = \oneop$ if and only if
$\eps_{^u\rho,^u\sigma} = \oneop$. In other words $u\in\cU(\M_0)$,
$\rho\mapsto{^u}\rho$ gives rise to a group of UBTC autoequivalences
of $\C$ which are also strict tensor and automorphic.

It can happen that $upu^* = p$, e.g., if $u$ is localizable away from
the cut geometric point $p\in I_0$. Otherwise $u$ and $p$ need not
\lq\lq commute" and $upu^*$ can be viewed as a \lq\lq fat" point of $\M_0$.  
\end{example}

\begin{example}\label{ex:fuzzypts}
Let $p=(\A(I_1),\A(I_2))$ as in the previous example and consider the modular group of $\M_0$ with respect to any faithful normal state $\varphi$, e.g., the vacuum state $\omega(\cdot)=(\Omega|\cdot\Omega)$ of $\{\A\}$. Denote by $\Delta_{\varphi}$ and $\sigma_t^{\varphi} = \Ad_{\Delta_{\varphi}^{it}}$, $t\in\RR$ respectively the modular operator and the modular group of ($\M_0,\varphi$). Then $\Delta_{\varphi}^{it} p \Delta_{\varphi}^{-it}$ is an abstract point of $\M_0$, for every $t\in\RR$. Indeed (i) and (ii) follow as before, while (iii) is guaranteed by the existence of localizable Connes cocycles $u_{\rho,t}\in\cU(\M_0)$, as shown by \cite[Prop.\ 1.1]{Lon97}, which fulfill the intertwining relation $^t\rho = \Ad_{u_{\rho,t}} \rho$ on $\M_0$ for $^t\rho := \sigma_t^{\varphi}\circ\,\rho\circ\sigma_{-t}^{\varphi}$. Hence $^t\rho$ is again $\DHR$ and $t\mapsto {^t\rho}$ gives a tensor autoequivalence of $\C$, defined on arrows as $^t s := \sigma_{t}^{\varphi}(s)$. Using more advanced technology we can show that $t\mapsto {^t\rho}$ respects the braiding of $\C$. Namely
$$\eps_{{^t\rho,^t\sigma}} = u_{\sigma,t} \sigma(u_{\rho,t}) \eps_{\rho,\sigma} \rho(u_{\sigma,t}^*) u_{\rho,t}^* = u_{\sigma\rho,t} \eps_{\rho,\sigma} u_{\rho\sigma,t}^* = \sigma_t^\varphi (\eps_{\rho,\sigma}) = {^t \eps_{\rho,\sigma}}$$
where the first equality follows by naturality of the braiding, the
second and third by tensoriality and naturality of the Connes cocycles
associated to the modular action of $\RR$, see respectively
\cite[Prop.\ 1.4, 1.3]{Lon97}. In particular, $\eps_{\rho,\sigma} =
\oneop$ if and only if $\eps_{^t\rho,^t\sigma} = \oneop$, hence
condition (iv) is satisfied. As before $t\in\RR$,
$\rho\mapsto{^t}\rho$ gives rise to a group of UBTC autoequivalences
of $\C$ which are again strict tensor and automorphic. The point
$\Delta_{\varphi}^{it} p \Delta_{\varphi}^{-it}$ is not honest in general, but highly fuzzy.

In the special case of the vacuum state $\varphi = \omega$, the
modular action is geometric and coincides with the dilations subgroup
$t\mapsto\Lambda_{I_0}^t$ of $\Mob$ which preserve $I_0$
(Bisognano-Wichmann property \cite[Prop.\ 1.1]{GuLo96}), hence
$\Delta_{\omega}^{it} p \Delta_{\omega}^{-it} = \Lambda_{I_0}^{-2\pi
t}(p)$ is just a M\"obius transformed honest point (with respect to $\{\A\}$).
\end{example}

In the terminology of \cite[App.\ 5]{Tur10mueger} due to M. M\"uger,
see also \cite[App.\ A]{Lon97}, we have found that $\cU(\M_0)$ (and
all of its subgroups) and $\RR$ (for every choice of faithful normal
state on $\M_0$) \emph{act} on $\C$ (as UBTC strict automorphisms),
and the actions are strict. One can then define the category of
\lq\lq$G$-fixed points", $\C^G$, where $G$ denotes one of these groups
with the associated action. In our case $\C^G = \C$ because all the
objects $\rho$ of $\C$ are \lq\lq $G$-equivariant", i.e., admit a
\emph{cocycle} for the $G$-action, i.e., unitary isomorphisms
$v_{\rho,g}:\rho\rightarrow {^g}\rho$, $g\in G$, such that
$v_{\rho,gh} = {^g}(v_{\rho,h})\circ v_{\rho,g}$. In Example
\ref{ex:fatpts} the cocycle identity follows because $\rho$ are
*-homomorphisms, in Example \ref{ex:fuzzypts} it coincides with the
characterization of the Connes cocycles.

In our case these actions are also implemented by unitaries $U_g\in\cU(\Hilb)$, hence we have examples of (groups of) automorphisms of the braided action $\C\hookrightarrow \End(\M_0)$ in the sense of Definition \ref{def:bractionisom}.

\section{Prime UMTCs and prime conformal nets}\label{sec:primeCFTline}

There are other types of abstract points, living inside completely rational nets that \emph{factorize} as tensor products, which are abstract but neither honest nor fuzzy, in the sense that they are almost geometric, or better, geometric in 1+1 dimensions. 
Ruling out these cases will lead us to the notion of \emph{prime conformal nets}.

\begin{example}\label{ex:tensorproductnet}
Consider a completely rational conformal net on the line of the form
$\{I\in\cI \mapsto \A(I) = \A_1(I) \otimes \A_2(I)\} = \{\A_1 \otimes
\A_2\}$, where $\{\A_1\}$, $\{\A_2\}$ are two nontrivial nets, then $\DHR\{\A\} \simeq \DHR\{\A_1\}\boxtimes\DHR\{\A_2\}$ as UBTCs. An equivalence is given by $\rho\boxtimes\sigma \mapsto \rho\otimes\sigma$, $T\boxtimes S \mapsto T\otimes S$ where essential surjectivity follows from \cite[Lem.\ 27]{KLM01} and the braiding on the l.h.s.\ is defined as $\eps_{\rho\boxtimes\sigma, \tau\boxtimes\eta} = \eps^{\A_1}_{\rho, \tau} \boxtimes \eps^{\A_2}_{\sigma, \eta}$.
We can consider as before a local algebra $\M_0 := \A_1(I_0)\otimes\A_2(I_0)$ for some interval $I_0\in\cI$, and take two honest points 
$p_1 = (\A_1(I_1), \A_1(I_2))$ in $\A_1(I_0)$ and $p_2 = (\A_2(J_1),\A_2(J_2))$ in $\A_2(I_0)$ respectively in the two components. Now setting $\N := \A_1(I_1)\otimes\A_2(J_1)$
we have that irreducibles in $\C_\N$ are given by $\Ad_u\, \rho\otimes\sigma$ for some $\rho\in\DHR^{I_1}\{\A_1\}$, $\sigma\in\DHR^{J_1}\{\A_2\}$ and $u\in\cU(\N)$.
Moreover, the pair of algebras $q = (\N,\N^c)$ is an abstract point of $\M_0$, but \emph{not} honest unless $I_1 = J_1$. In other words, $q = p_1\otimes p_2$ is an honest point of $\M_0$ if and only if $p_1 = p_2$ as geometric points of $I_0$. 
\end{example}

We recall the following definition due to \cite{Mue03}, see also \cite{DMNO13}.

\begin{definition}\label{def:primeUMTC}
A UMTC $\C$ is called a \textbf{prime UMTC} if $\C\not\simeq\catVec$ and every full unitary fusion subcategory $\D\subset\C$ which is again a UMTC is either $\D\simeq\C$ or $\D\simeq\catVec$ as UBTCs.
\end{definition}

The terminology is motivated by the following
proposition, which is among the deepest results on the structure of UMTCs.
It establishes prime UMTCs as building blocks in the classification program of UMTCs, see \cite{RSW09}. 

\begin{proposition}\label{prop:primefactorizUMTCs} \emph{\cite{Mue03}, \cite{DGNO10}.}
Let $\C$ be a UMTC, let $\D\subset\C$ be a unitary full fusion
subcategory and consider the centralizer of $\D$ in $\C$ \footnote{or
  braided relative commutant of $\D\subset\C$. Cf.\ the definition of relative commutant $\D^c$ we introduced in Section \ref{sec:dualityrelations} for full inclusions of tensor categories. Cf.\ also the definition \cite[Def. 2.9]{HePe15} of relative commutant in the sense of Drinfeld.} defined as the full subcategory of $\C$ with objects
$$\Z_{\C}(\D) := \big\{x\in\C : \eps_{x,y} = \eps^\op_{x,y}\, ,\,y\in\D\big\}.$$
It holds
\begin{itemize}
\item $\Z_{\C}(\D)$ is a unitary (full) fusion subcategory of $\C$, which is also replete, and $\Z_{\C}(\Z_{\C}(\D)) = \overline \D$ where $\overline \D$ denotes the repletion of $\D$ in $\C$.
\end{itemize}
If $\D$ is in addition a UMTC, i.e., $\Z_{\D}(\D)\simeq\catVec$, then
\begin{itemize} 
\item $\Z_{\C}(\D)$ is also a UMTC and $\C \simeq \D \boxtimes \Z_{\C}(\D)$ as UBTCs.
\end{itemize}
In particular, every UMTC admits a finite prime factorization, i.e.
$$\C\simeq\D_1\boxtimes\ldots\boxtimes\D_n$$
as UBTCs, where $\D_i$, $i=1,\ldots,n$ are prime UMTCs, fully realized in $\C$.
\end{proposition}

\begin{remark}
Observe that assuming $\DHR\{\A\}$ to be prime as an abstract UMTC
rules out holomorphic nets. Moreover the examples seen in
\ref{ex:tensorproductnet} cannot arise, unless one of the two tensor
factors is holomorphic, i.e., $\{\A\} = \{\A_1\otimes \Aholo\}$. The
following definition is aimed to rule out also
this case.
\end{remark}

\begin{definition}\label{def:primeconfnet}
Let $\{\A\}$ be a completely rational conformal net on the line. Fix arbitrarily $I_0\in\cI$ and let $\M_0 = \A(I_0)$, $\C = \DHR^{I_0}\{\A\}$. We call $\{\A\}$ a \textbf{prime conformal net} if the following conditions are satisfied.
\begin{itemize} \itemsep0mm
\item $\C\simeq\DHR\{\A\}$ is a prime UMTC.
\item For every ordered pair $p=(\N,\N^c)$, $q=(\M,\M^c)$ of abstract points of $\M_0$, if $\N\vee\M^c$ is normal in $\M_0$ then $\M\subset\N$, in particular $\N\vee\M^c = \M_0$.
\end{itemize}
\end{definition}

\begin{remark}
Notice that the primality assumption on $\C \simeq \DHR\{\A\}$ is purely categorical, i.e., invariant under equivalence of UBTCs, hence contains no information about the actual size of the category. By definition of prime UMTCs, holomorphic nets are \emph{not} prime conformal nets.
\end{remark}

\begin{remark}
If $p$, $q$ mutually fulfill, e.g., $\cR = (\cR \cap \cS) \vee (\cR
\cap \cS^c)$ for $\cR,\cS\in\{\N,\N^c,\M,\M^c\}$ (resembling strong additivity), then the
statements $\M\subset\N$ and $\N\vee\M^c = \M_0$ are actually equivalent. 
\end{remark}

It is easy to see that \emph{prime} conformal nets cannot factor through nontrivial holomorphic subnets.

\begin{example}\label{ex:tensorproductwithholonet}
Let $\{\A\}$ be a prime conformal net on the line, hence not holomorphic, but factoring through a holomorphic subnet, $\{\A\}=\{\A_1\otimes\Aholo\}$. Considering points $p_1\otimes p_2$ of $\M_0$ like in Example \ref{ex:tensorproductnet}, it is easy to construct $\N\vee\M^c$ which are normal in $\M_0$ but neither exhaust $\M_0$ nor have $\M\subset\N$, e.g., enlarging $\M$ in the holomorphic component. Then $\{\A\}$ cannot be prime unless $\{\Aholo\} = \{\CC\}$.
\end{example}

\begin{remark}
Both the notion of primality for completely rational conformal nets and the property of not factorizing through holomorphic subnets are 	\emph{invariant} under isomorphism of nets.
\end{remark}

Concerning the converse of the implication seen in Example \ref{ex:tensorproductwithholonet}, let $\{\A\}$ be a completely rational net, not necessarily prime, take $p$, $q$ as in Definition \ref{def:primeconfnet}. The idea is that $(\N\vee\M^c)^c = \N^c\cap \M$ are abstract \lq\lq interval algebras'' which lie in the \lq\lq holomorphic part'' of the net whenever $\N\vee\M^c$ is normal in $\M_0$. More precisely, we can show that they necessarily factor out in a tensor product subalgebra of $\M_0$, and that the local subcategories associated to them \`{a} la $\DHR$ are trivial, namely $\C_{\N^c}\cap\C_\M\subset\catVec$. \footnote{We identify $\catVec$ with the full subcategory of $\C$ whose objects are the \emph{inner} endomorphisms, cf.\ Proposition \ref{prop:dualofsergio}.}

\begin{proposition}\label{prop:holotensorsplit}
Let $\{\A\}$ be a completely rational conformal net on the line, fix $I_0\in\cI$ and let $\M_0 = \A(I_0)$, $\C = \DHR^{I_0}\{\A\}$. Consider the family $\cF$ of ordered pairs of abstract points $p = (\N,\N^c)$, $q = (\M,\M^c)$ such that $\N\vee\M^c$ is normal in $\M_0$, then the following holds.
\begin{itemize} \itemsep0mm
\item For every $(p,q) \in\cF$ we have $\C_{\N^c}\cap\C_\M\subset\catVec$.
\item Consider the subalgebra of $\M_0$ defined as 
$$\M_0^{\holo} := \bigvee_{(p,q)\in\cF} \N^c \cap \M$$ 
then $\M_0^{\holo}$ is either $\CC$ or a type $\III_1$ subfactor of $\M_0$, and the same holds for the relative commutant 
$$(\M_0^{\holo})^c = \bigcap_{(p,q)\in\cF} \N \vee \M^c.$$ 
Moreover we have a splitting
$$\M_0^{\holo} \vee (\M_0^{\holo})^c \cong \M_0^{\holo} \otimes (\M_0^{\holo})^c$$
as von Neumann algebras.
\end{itemize}
\end{proposition}

\begin{proof}
Normality of $\N\vee\M^c$ in $\M_0$ means $\N\vee\M^c =
(\N\vee\M^c)^{cc}$, equivalently $(\N^c\cap\M)^c = \N\vee\M^c$, but
there is a more useful characterization. Without assuming normality,
let $\rho\in\C_\N$, $\tilde\rho\in\C_{\M^c}$ and $u$ a unitary charge
transporter from $\rho$ to $\tilde\rho$. For every $a\in\N^c\cap\M$ we
have $ua = u \rho(a) = \tilde\rho(a) u = a u$ hence
$u\in(\N^c\cap\M)^c = (\N\vee\M^c)^{cc}$. Denoting by 
$$\cU_{\C} (\N,\M^c) := \vN\{u\in\Hom_{\C}(\rho,\tilde\rho) \cap \cU(\M_0),\; \rho\in\C_\N, \; \tilde\rho\in\C_{\M^c}\}$$
the von Neumann algebra generated by the charge transporters, we have
\begin{equation}\label{eq:abstrchargetranspos}\N\vee\M^c \subset \cU_{\C} (\N,\M^c) \subset (\N\vee\M^c)^{cc}\end{equation}
where the first inclusion holds because the unitaries in $\cU(\N)$ and $\cU(\M^c)$ generate inner automorphisms from the vacuum. Normality of $\N\vee\M^c$ in $\M_0$ turns out to be \emph{equivalent} to $\cU_{\C} (\N,\M^c) = \cU_{\C} (\N,\M^c)^{cc} = \N\vee\M^c$.
Using this we can show that $\C_{\N^c}\cap\C_\M\subset\catVec$. Let $\rho \in \C_{\N^c}\cap\C_\M$ and observe that $\C_{\N^c}\cap\C_\M = \N^\perp \cap {\M^c}^\perp = (\N \vee \M^c)^\perp$ because endomorphisms in $\C$ are normal. Now by normality of $\N\vee\M^c$ in $\M_0$ we have that $\rho\in\cU_\C(\N,\M^c)^\perp$, i.e., $\rho(u) = u$ for every unitary generator $u\in\cU_\C(\N,\M^c)$. On the other hand for every $\sigma\in\C_\N$ and $\tilde\sigma := \Ad_u \sigma\in\C_{\M^c}$ we have $\eps_{\rho,\tilde\sigma} = \oneop$ by assumption (iv), i.e., $\rho(u) = u \eps_{\rho,\sigma}$ by naturality of the braiding, hence $\eps_{\rho,\sigma} = \oneop$. Again by (iv) we have $\eps_{\sigma,\rho} = \oneop$ and by (iii) $\C_\N\simeq\C$ from which we can conclude that $\rho$ has vanishing monodromy with every sector, hence $\rho\in\catVec$ by modularity of $\C$, showing the first statement.

The second statement follows using modular theory on abstract points
of $\M_0$, see Example \ref{ex:fuzzypts}, \cite[Prop.\
2.8]{Reh00-1}. Let $\sigma_t^{\omega} := \Ad_{\Delta_\omega^{it}}$,
$t\in\RR$ be the modular group of $\M_0$ associated to the vacuum
state $\omega$ of the net, we know that if $p$ is an abstract point of
$\M_0$ then $\sigma_t^{\omega}(p)$, $t\in\RR$ are also abstract
points. Furthermore $t\mapsto\sigma_t^\omega$ respects $\M_0$ and the
normality property for subalgebras of $\M_0$, hence maps $\cF$ onto
$\cF$ because $(\sigma_t^\omega)^{-1} = \sigma_{-t}^\omega$ and we
conclude $\sigma_t^\omega (\M_0^{\holo}) = \M_0^{\holo}$,
$t\in\RR$. By Takesaki's theorem \cite{Tak72} we have a faithful
normal conditional expectation $E:\M_0\rightarrow\M_0^{\holo}$
intertwining $E\circ \sigma_t^\omega = \sigma_t^\varphi\circ E$,
$t\in\RR$, where $\varphi$ is the faithful normal state obtained by
restricting $\omega$ to $\M_0^{\holo}$ and $\sigma_t^\varphi$ is the
associated modular group, see \cite[Sec. 10]{Str81}. Now the vacuum
state $\omega$ is given by the unique vector invariant under the group
of $I_0$-preserving dilations by \cite[Cor.\ B.2]{GuLo96}. This,
together with the Bisognano-Wichmann property \cite[Prop.\
1.1]{GuLo96}, imply that $t\mapsto\sigma_t^\omega$ is ergodic on
$\M_0$, hence $t\mapsto\sigma_t^\varphi$ is ergodic on
$\M_0^{\holo}$. In other words, $\varphi$ has trivial centralizer,
then by \cite[Prop.\ 6.6.5]{Lon2} $\M_0^{\holo}$ is a \emph{factor} of
type $\III_1$ or trivial $\M_0^{\holo} = \CC$. The same holds for
$(\M_0^{\holo})^c$. In particular, $\M_0^{\holo}$ being a subfactor of $\M_0$, we can apply \cite[Cor.\ 1]{Tak72} to get the splitting of $\M_0^{\holo} \vee (\M_0^{\holo})^c$ as von Neumann tensor product, completing the proof of the second statement.
\end{proof}

\section{Comparability of abstract points}\label{sec:comparability}

In the previous sections we analysed the braiding condition (iv) in
Definition \ref{def:apts}: 
$\eps_{\rho,\sigma} = \oneop$ on honest and abstract points of a net
$\{\A\}$, see Eq.\ (\ref{eq:QFTbraiding}), Lemma
\ref{lem:abstrchargetranspo}, Proposition
\ref{prop:naturalityiefinmanyeqs}, and showed how it
can be led far away from geometry in Section \ref{sec:fuzzyapts}.

In this section we draw some of its consequences, as in the proof
Proposition \ref{prop:holotensorsplit}, and to do so we introduce
\emph{comparability} $p\sim q$ of abstract points, along with
an order relation $p<q$ compatible with the geometric ordering of 
honest points. The terminology is motivated by the fact that two abstract 
points $p\sim q$ in a \emph{prime} conformal net are necessarily 
$p<q$ or $q<p$ or $p=q$, see Proposition \ref{prop:totalorderapts}. 
The order symbols should be intended as inclusions of relative 
complement algebras of $p,q$ in $\M_0$.

Let $p = (\N,\N^c)$, $q = (\M,\M^c)$ be two abstract points of $\M_0$ as in Definition \ref{def:apts} and $(\cR,\cS)$ be any pair of elements from $\{\N,\N^c,\M,\M^c\}$. Similarly to Eq.\ (\ref{eq:abstrchargetranspos}) we have that the von Neumann algebras of unitary charge transporters
\begin{equation}\label{eq:vNtranspo}\cU_{\C} (\cR,\cS) := \vN\{u\in \Hom_{\C}(\rho,\tilde\rho)\cap\cU(\M_0),\; \rho\in\C_\cR, \; \tilde\rho\in\C_\cS\}\end{equation}
always sit in between
$$\cR\vee\cS \subset \cU_{\C} (\cR,\cS) \subset (\cR\vee\cS)^{cc},$$
in particular $\cU_{\C} (\cR,\cS) ^{cc} = (\cR\vee\cS)^{cc}$. Hence
asking \emph{normality} of (\ref{eq:vNtranspo}) in $\M_0$ is
equivalent to asking that charge transporters \emph{generate} as von
Neumann algebras the relative commutants, cf.\ \cite[Cor.\
4.3]{Mue99}, \cite[Thm.\ 33]{KLM01}, i.e., $\cU_{\C} (\cR,\cS) =
(\cR\vee\cS)^{cc} = (\cR^c\cap\cS^c)^c$.
 
Notice that, e.g., $\cU_{\C} (\N,\N)$ and $\cU_{\C} (\N,\N^c)$ are always normal in $\M_0$ by (ii) and that $\cU_{\C} (\cR,\cS) = \cU_{\C} (\cS,\cR)$ by definition.

\begin{lemma}\label{lem:intcats}
In the above notation, assume that $\cU_\C(\cR,\cS)$ is normal in $\M_0$ for every pair $(\cR,\cS)$ of elements in $\{\N,\N^c,\M,\M^c\}$, then
\begin{itemize} \itemsep0mm
\item $\C_{\N\cap\M} = \C_\N \cap \C_\M$ and $\C_{\N^c\cap\M^c} = \C_{\N^c} \cap \C_{\M^c}$.
\item $\C_{\N\cap\M^c} \subset \C_\N \cap \C_{\M^c}$ and $\rho\in\C_{\N\cap\M^c}$ if and only if $\rho$ is an inner endomorphism of $\C$; in symbols: $\C_{\N\cap\M^c} = (\C_\N \cap \C_{\M^c}) \cap \catVec$. Similarly $\C_{\M\cap\N^c} = (\C_\M \cap \C_{\N^c}) \cap \catVec$.
\end{itemize}
\end{lemma}

\begin{proof}
Consider the intersection of left-left relative complements $\C_\N \cap \C_\M$. The inclusion $\C_{\N\cap\M} \subset \C_\N \cap \C_\M$ reads ${(\N\cap\M)^c}^\perp \subset {\N^c}^\perp \cap {\M^c}^\perp = (\N^c \vee \M^c)^\perp$ hence follows easily by taking duals of $\N^c \vee \M^c \subset (\N^c \vee \M^c)^{cc} = (\N\cap\M)^c$. The opposite inclusion follows from the braiding condition and normality assumption on charge transporters. Take $\rho\in\C_\N \cap \C_\M$ then by (iv) we have $\eps_{\rho,\tilde\sigma} = \oneop$ for every $\tilde\sigma := \Ad_u \sigma\in\C_{\M^c}$ where $\sigma\in\C_{\N^c}$ and $u$ is a unitary generator of $\cU_\C(\N^c,\M^c)$. Hence $\rho(u) = u\eps_{\rho,\sigma}$ by naturality of the braiding. But also $\eps_{\rho,\sigma} = \oneop$ by assumption (iv) and $\rho\in\cU_\C(\N^c,\M^c)^\perp = (\N\cap\M)^{c\perp}$ follows, hence we have the first statement. The right-right case follows similarly.

In the left-right case the inclusion $\C_{\N\cap\M^c} \subset \C_{\N} \cap \C_{\M^c}$ can be proper, as shown by Proposition \ref{prop:dualofsergio} in the honest case. Take $\rho\in\C_{\N} \cap \C_{\M^c}$, by normality $\rho\in\C_{\N \cap \M^c}$ if and only if $\rho(u) = u$ for every unitary generator $u\in\cU_\C(\N^c,\M)$. But now by (iv) we have $\eps_{\tilde\sigma,\rho} = \oneop$ for every $\tilde\sigma := \Ad_u \sigma\in\C_\M$ where $\sigma\in\C_{\N^c}$, $u\in\cU_\C(\N^c,\M)$, hence $\rho(u) = u \eps_{\sigma,\rho}^*$ together with $\eps_{\rho,\sigma} = \oneop$. By assumption (iii) $\C_{\N^c}\simeq\C$ and modularity of $\C$, we can conclude that $\rho\in\C_{\N\cap\M^c}$ if and only if $\rho\in\catVec$, and the proof is complete.
\end{proof}

As already remarked, given a pair of abstract points $p=(\N,\N^c)$,
$q=(\M,\M^c)$ of $\M_0$, the algebras $\N\cap\M^c$ can be viewed as
abstract \lq\lq interval algebras'' of $\M_0$ with associated \lq\lq
local" $\DHR$ subcategories $\C_{\N}\cap\C_{\M^c}$.
 
Denote by $\Delta(\C)$ the \emph{spectrum} of $\C$ and let $\cU_{\C_{\N^c}\cap\C_\M} (\N,\M^c)\subset \cU_\C (\N,\M^c)$ be the subalgebra generated by $\rho$-charge transporters associated to sectors $[\rho]\in\Delta(\C_{\N}\cap\C_{\M^c})$. The vacuum $[\id]$ is always in the spectrum, hence $\cU_{\C_{\N^c}\cap\C_\M} (\N,\M^c)$ is also intermediate in $\N\vee\M^c \subset(\N\vee\M^c)^{cc}$. 

\begin{lemma}\label{lem:modularityintcats}
In the above notation, assume that $\cU_{\C_{\N^c}\cap\C_\M} (\N,\M^c)$ and $\cU_{\C_{\M^c}\cap\C_\N} (\M,\N^c)$ are normal in $\M_0$, then $\C_{\N^c}\cap\C_\M$ and $\C_{\M^c}\cap\C_\N$ have \lq\lq modular spectrum'', i.e.
$$\Z_{\C_{\N^c}\cap\C_\M}(\C_{\N^c}\cap\C_\M) \subset \catVec,\quad \Z_{\C_{\M^c}\cap\C_\N}(\C_{\M^c}\cap\C_\N) \subset \catVec.$$
\end{lemma}

\begin{proof}
Let $\rho\in\C_{\N^c} \cap \C_\M$ such that $\eps_{\rho,\sigma} = \eps_{\rho,\sigma}^\op$ for all $\sigma\in\C_{\N^c} \cap \C_\M$. Inspired by \cite[Lem.\ 3.2]{Mue99} we can write $\eps_{\rho,\sigma} = u^*\rho(u)$ and $\eps_{\rho,\sigma}^\op = x^*\rho(x)$ where $u$ and $x$ are unitaries transporting $\sigma$ respectively to $\C_{\M^c}$ and $\C_{\N}$, see Lemma \ref{lem:abstrchargetranspo}. Hence triviality of the monodromy $\eps_{\rho,\sigma} = \eps_{\rho,\sigma}^\op$ is triviality of the action $\rho(ux^*) = ux^*$. Moreover every generator $w$ of $\cU_{\C_{\N^c}\cap\C_\M} (\N,\M^c)$ can be written as $w=ux^*$ with $u$ and $x$ as above. By normality $\cU_{\C_{\N^c}\cap\C_\M} (\N,\M^c) = (\N\vee\M^c)^{cc}$ hence, reversing the argument, one can drop the restriction $\sigma\in\C_{\N^c} \cap \C_\M$
and get $\eps_{\rho,\sigma} = \eps_{\rho,\sigma}^\op$ for all $\sigma\in\C$. By modularity of $\C$ we get $\rho\in\catVec$. Analogously interchanging $\N$ and $\M$. 
\end{proof}

Normality of $\cU_{\C_{\N^c}\cap\C_\M} (\N,\M^c)$ obviously implies normality of $\cU_{\C} (\N,\M^c)$.
We are now ready to introduce the notion of comparability of two abstract points $p,q$ mentioned in the beginning of this section.

\begin{definition}\label{def:comparapts}
Let $\{\A\}$ be a completely rational conformal net on the line. In the notation of Definition \ref{def:apts}, two abstract points $p = (\N,\N^c)$, $q = (\M,\M^c)$ of $\M_0$ are called \textbf{comparable} if they fulfill the following
\begin{itemize} \itemsep0mm
\item $\cU_{\C_{\cR^c}\cap\C_{\cS^c}}(\cR,\cS) = {\cU_{\C_{\cR^c}\cap\C_{\cS^c}}(\cR,\cS)}^{cc}$.
\item $\cR\vee\cS = (\cR\vee\cS)^{\perp\perp}$.
\end{itemize}
for every pair $(\cR,\cS)$ in $\{\N,\N^c,\M,\M^c\}$. In this case, we write $p \sim q$.
\end{definition}

Observe that $\cU_{\C_{\cR^c}\cap\C_{\cS^c}}(\cR,\cS)$ and $(\C_{\cR^c}\cap\C_{\cS^c})^\perp = (\cR\vee\cS)^{\perp\perp}$ are both intermediate algebras in the inclusions $\cR\vee\cS \subset (\cR\vee\cS)^{cc}$. Hence comparability means that these bounds are maximally, respectively minimally, saturated.

\begin{remark}\label{rmk:honestarecomp}
We have already motivated the normality condition on charge
transporters. Concerning biduality, it easily holds for left or right
local half-line algebras, see Proposition \ref{prop:dualofwedge},
Remark \ref{rmk:dualofwedgecat}, and for two-interval algebras, as we
have shown in Proposition \ref{prop:dualoflocalcat}. Notice also that
comparability is manifestly reflexive, symmetric and invariant under
isomorphism of nets (but not manifestly transitive). 
\end{remark}

\begin{proposition}\label{prop:totalorderapts}
Let $\{\A\}$ be a prime conformal net on the line (Definition
\ref{def:primeconfnet}) and take two abstract points $p = (\N,\N^c)$,
$q = (\M,\M^c)$ of $\M_0$. If $p\sim q$ then either $p<q$ or $q<p$ or
$p=q$, i.e., respectively $\N\subset\M$ or $\M\subset\N$ or $\N=\M$.

In particular, in the case of a prime conformal net, comparability of $p$ and $q$ can be checked on the two pairs $(\N,\M^c)$, $(\M,\N^c)$.
\end{proposition}

\begin{proof}
The idea of the proof is that $\N^c\cap\M$ and $\M^c\cap\N$ are, a priori, abstract interval algebras of two \emph{different} tensor factors of the net. Call for short $\C_1 := \C_{\N^c} \cap \C_\M$ and $\C_2 := \C_{\M^c} \cap \C_\N$ and observe that
\begin{equation}\label{eq:C1inC2c}\C_1 \subset \Z_{\C}(\C_2), \quad \C_2 \subset \Z_{\C}(\C_1)\end{equation}
because for every $\rho\in\C_1$, $\sigma\in\C_2$ we have $\eps_{\rho,\sigma} = \oneop =$ and $\eps_{\sigma,\rho} = \oneop$ by condition (iv), in particular $\eps_{\sigma,\rho}\eps_{\rho,\sigma} = \oneop$. We also have
\begin{equation}\label{eq:C1capC1cinVec}\Z_{\C_1}(\C_1)\subset\catVec, \quad \Z_{\C_2}(\C_2)\subset\catVec\end{equation}
by Lemma \ref{lem:modularityintcats}. Notice that it can be
$\C_1=\C_2=\{\id\}$, e.g., if $\N=\M$. In order to invoke primality of
the $\DHR$ category $\C$ as a UMTC, we take the closures of $\C_1,
\C_2\subset\C$ under conjugates, subobjects, finite direct sums,
tensor products and unitary isomorphism classes. Denote them
respectively by $\tilde\C_1$, $\tilde\C_2$. In other words, they are
the smallest replete fusion subcategories of $\C$ containing $\C_1$,
$\C_2$ respectively. Thanks to \cite[Thm.\ 3.2]{Mue03}, see also
\cite[Thm.\ 3.10]{DGNO10}, they are characterized as double braided
relative commutant subcategories of $\C$, i.e.
$$\tilde\C_1 = \Z_{\C}(\Z_\C(\tilde\C_1)),\quad \tilde\C_2 = \Z_{\C}(\Z_\C(\tilde\C_2)).$$
Now inclusions (\ref{eq:C1inC2c}) and (\ref{eq:C1capC1cinVec}) clearly
extend to subobjects, direct sums, tensor products and unitary
isomorphism classes, because the vanishing of the monodromy is a
condition stable under such operations, see \cite[Sec. 2.2]{Mue00},
and $\catVec$ is a replete fusion subcategory of $\C$. We need to
check that (\ref{eq:C1inC2c}) and (\ref{eq:C1capC1cinVec}) extend to
conjugates because neither of the two sides of (\ref{eq:C1inC2c}) nor
the l.h.s.\ of (\ref{eq:C1capC1cinVec}) are a priori rigid. Let
$\rho\in\C_1$, $\sigma\in\C_2$ and choose a conjugate $\rhobar\in\C$
of $\rho$, we want to show that
$\eps_{\sigma,\rhobar}\eps_{\rhobar,\sigma} = \oneop$. By condition
(iii) we can assume $\rhobar\in\C_{\N^c}$ up to unitary isomorphism,
equivalently we could have assumed $\rhobar\in\C_{\M}$. By Proposition
\ref{prop:abstractresfun} we have that every solution of the conjugate
equations $R\in\Hom_{\C}(\id,\rhobar\rho)$, $\overline
R\in\Hom_{\C}(\id,\rho\rhobar)$ for $\rho$, $\rhobar$, see
\cite[Sec. 2]{LoRo97}, lies in $\N^c$, in particular $\sigma(R) = R$,
$\sigma(\overline R) = \overline R$. Hence we get
$\eps_{\rhobar,\sigma} = R^* \rhobar(\eps_{\rho,\sigma}^*)
\rhobar\sigma(\overline R) = R^* \rhobar(\overline R) = \oneop$ and
similarly $\eps_{\sigma,\rhobar} = \rhobar\sigma(\overline R^*)
\rhobar(\eps_{\sigma,\rho}^*) R = \rhobar(\overline R^*) R =
\oneop$. In particular, $\rhobar$ and $\sigma$ have vanishing
monodromy.
 
Summing up we have $\tilde\C_1 \subset \Z_\C(\C_2)$ and similarly $\tilde\C_2 \subset \Z_\C(\C_1)$. Moreover, given $\sigma\in\C_2$ choose a conjugate $\sigmabar\in\C$ and observe that the vanishing of the monodromy of $\sigmabar$ and every $\rho$ in $\tilde\C_1$ is equivalent to the vanishing of the monodromy of $\sigma$ and every $\rho$, by rigidity of $\tilde\C_1$, see \cite[Eq.\ (2.17)]{Mue00}. Hence we have  
\begin{equation}\label{eq:C1inC2creplete}\tilde\C_1 \subset \Z_\C(\tilde\C_2),\quad \tilde\C_2 \subset \Z_\C(\tilde\C_1)\end{equation}
and the two inclusions are equivalent by the double braided relative commutant theorem. We can extend also inclusions (\ref{eq:C1capC1cinVec}) by observing that $\Z_{\C_1}(\tilde\C_1) \subset \Z_{\C_1}(\C_1)\subset\catVec$ and that, given $\rho\in\C_1$ and a conjugate $\rhobar\in\C$, the vanishing of the monodromy of $\rhobar$ and every $\sigma$ in $\tilde\C_1$ is equivalent, as above, to the vanishing of the monodromy of $\rho$ and every $\sigma$. Thus we have $\rho\in\catVec$, hence $\rhobar\in\catVec$, and we conclude
\begin{equation}\Z_{\tilde\C_1}(\tilde\C_1) = \catVec, \quad \Z_{\tilde\C_2}(\tilde\C_2) = \catVec\end{equation}
which means modularity for the replete fusion subcategories $\tilde\C_1, \tilde\C_2\subset\C$. By primality of $\C$ as a UMTC, see Definition \ref{def:primeUMTC}, the two subcategories are either $\C$ or $\catVec$ and by the inclusions (\ref{eq:C1inC2creplete}) we can assume $\tilde\C_1 = \catVec$, up to exchanging the roles of $\N$ and $\M$.
 
In particular, we obtain $\C_1 = \C_{\N^c}\cap\C_\M \subset \catVec$, hence
$$\C_{\N^c\cap\M} = \C_{\N^c}\cap\C_\M$$
by Lemma \ref{lem:intcats}, i.e., ${(\N^c\cap\M)^c}^\perp = (\N \vee \M^c)^\perp$. 
Now by comparability we have a biduality relation $(\N \vee
\M^c)^{\perp\perp} = \N \vee \M^c$, while
${(\N^c\cap\M)^c}^{\perp\perp} = {(\N^c\cap\M)^c}$ follows by the same
argument as in Proposition \ref{prop:4tuple}. By taking
duals we have that $\N \vee \M^c$ is normal in $\M_0$, hence
$\M\subset\N$ by the primality assumption on the net. In particular, $\C_1 = \{\id\}$, and the proof is complete.
\end{proof}

As said before, normality of $\cU_{\C_{\N^c}\cap\C_\M} (\N,\M^c)$ is
equivalent to saying that the inclusion
$\N\vee\M^c\subset(\N\vee\M^c)^{cc}$ is generated by charge
transporters associated to sectors
$[\rho]\in\Delta(\C_{\N^c}\cap\C_\M)$. We could strengthen this
assumption by asking that the inclusion has the structure of a
Longo-Rehren inclusion associated with $\{[\rho]\in \Delta(\C_{\N^c}
\cap \C_{\M})\}$. This amounts to specifying not only the generators
of the extension, but also the algebraic relations among them
\cite[Eq.\ (15), Prop.\ 45]{KLM01}.

We show next that the latter can be derived, in our language of
abstract points, from the \emph{fusion} structure of the intersection
categories. However, we don't require, a priori, $\N\vee\M^c$ to split as a von Neumann tensor product, nor $\N$ and $\M^c$ to be commuting algebras.

\begin{proposition}
Let $\{\A\}$ be a completely rational conformal net on the line and take two abstract points $p=(\N,\N^c)$, $q=(\M,\M^c)$, in the notation of Definition \ref{def:apts}. If we assume that 
\begin{itemize} \itemsep0mm
\item $\cU_{\C_{\N^c}\cap\C_\M} (\N,\M^c)$ and $\cU_{\C_{\M^c}\cap\C_\N} (\M,\N^c)$ are normal in $\M_0$,
\item $\C_{\N^c}\cap\C_{\M}$ and $\C_{\M^c}\cap\C_{\N}$ are UFTCs in $\C$,
\item $\C_{\N}\cap\C_{\M} \simeq \C$ and $\C_{\N^c}\cap\C_{\M^c} \simeq \C$
\end{itemize}
then $\N\vee\M^c\subset (\N\vee\M^c)^{cc}$ and $\M\vee\N^c\subset (\M\vee\N^c)^{cc}$ have the structure of Longo-Rehren inclusions, in the sense that the generators of the extensions fulfill the relations \emph{\cite[Eq.\ (15)]{KLM01}}.
\end{proposition}

\begin{proof}
Consider the inclusion $\N\vee\M^c\subset (\N\vee\M^c)^{cc}$. Being
$\C_{\N^c}\cap\C_{\M}$ a UFTC we can arrange its irreducible sectors
$\{[\rho]\in\C_{\N^c}\cap\C_{\M}\}$ in a \emph{rational system}
$\{[\rho_i]\}_i$, in the terminology of \cite{KLM01}, see also
\cite{Reh90}, \cite{BEK99}. By assumption, for each $[\rho_i]$ we can
choose $\rhobar_i\in\C_{\N}\cap\C_{\M}$,
$\rho_i\in\C_{\N^c}\cap\C_{\M^c}$ and
$R_i\in\Hom_\C(\id,\rhobar_i\rho_i)$ such that $R_i^*R_i = d_{\rho_i}
\oneop$ and $R_0 = \oneop$. In particular, $R_i a = \rhobar_i\rho_i(a)
R_i$ for all $a\in\N\vee\M^c$ and $R_i\in(\N^c\cap\M)^c =
(\N\vee\M^c)^{cc}$.

Now, $R_i R_j \in \Hom_\C(\id,\rhobar_i\rho_i\rhobar_j\rho_j) = \Hom_\C(\id,\rhobar_i\rhobar_j\rho_i\rho_j)$ because, e.g., $\C_\N$ and $\C_{\N^c}$ commute in the sense of Proposition \ref{prop:cd}, and
$$R_iR_j = \sum_{k,\alpha,\beta} (w_\alpha w_\alpha^* \times v_\beta v_\beta^*) \cdot (R_i\times R_j)$$ 
where $k$ runs over irreducible components $[\rho_k]\prec[\rho_i][\rho_j]$ and $\alpha$, $\beta$ over orthonormal bases of isometries $w_\alpha\in\Hom_{\C_\N}(\rhobar_k,\rhobar_i\rhobar_j)$, $v_\beta\in\Hom_{\C_{\M^c}}(\rho_k,\rho_i\rho_j)$.
Then $\sum_{k,\alpha,\beta} w_\alpha w_\alpha^* \times v_\beta v_\beta^* \cdot R_i\times R_j = \sum_{k,\alpha,\beta} w_\alpha v_\beta\, \lambda_{\alpha,\beta}^k R_k$ where $\lambda_{\alpha,\beta}^k\in\CC$ because $[\rho_k]$ is irreducible, hence $[\id]\prec[\rhobar_k][\rho_k]$ with multiplicity one, and $\rhobar_k(v_\beta) = v_\beta$. Setting $C_{ij}^k := \sum_{\alpha,\beta} w_\alpha v_\beta\, \lambda_{\alpha,\beta}^k$ we have (non-canonical) intertwiners in $\Hom_\C(\rhobar_k\rho_k, \rhobar_i\rhobar_j\rho_i\rho_j) = \Hom_\C(\rhobar_k\rho_k, \rhobar_i\rho_i\rhobar_j\rho_j)$ which lie in $\N\vee\M^c$ and fulfill
$$R_iR_j = \sum_k C_{ij}^k R_k.$$
In particular, we have $C_{\overline i i}^0 \in \Hom_\C(\id, \rhobar_{\overline i}\rho_{\overline i}\rhobar_i\rho_i)$ again in $\N\vee\M^c$, hence
$R_{\overline i}^* C_{\overline i i}^0$ is a multiple of $R_i$, i.e., we get
$$R_i^* = \lambda C_{\overline i i}^{0*} R_{\overline i}$$
for some $\lambda\in\CC$, and we have shown up to normalization
constants the algebraic relations of \cite[Eq.\ (15)]{KLM01}.

On the other hand, by Frobenius reciprocity \cite[Lem.\ 2.1]{LoRo97} the $R_i$ generate the extension $\N\vee\M^c \subset (\N\vee\M^c)^{cc}$ because every unitary charge transporter $u\in\Hom_\C(\rho,\tilde\rho)$, $\rho\in\C_\N$, $\tilde\rho\in\C_{\M^c}$ such that $[\rho] = [\rho_i]$ for some $i$, can be written as $u = \lambda v \rho_i(r^*) R_{\overline i} = \lambda v r^* R_{\overline i}$ for suitable $\lambda\in\CC$, $v\in\M^c$ unitary and $r\in\N$ isometric. In particular, every $b\in(\N\vee\M^c)^{cc}$ admits a (not necessarily unique) \lq\lq harmonic'' expansion
\begin{equation}\label{eq:harmonicabstractexpansion}b = \sum_{i} b_i R_i\end{equation}
where $b_i\in\N\vee\M^c$, cf.\ \cite[Eq.\ (4.10)]{LoRe95}, \cite[Prop.\ 45]{KLM01}, and we are done.
\end{proof}

\begin{corollary}\label{cor:morethancomp}
With the assumptions of the previous proposition, $\N\vee\M^c$ is bidual in $\M_0$, i.e., $(\N\vee\M^c)^{\perp\perp} = \N\vee\M^c$. Moreover $\N\vee\M^c$ is normal in $\M_0$ if and only if $\C_{\N^c}\cap\C_{\M}\subset\catVec$, and $\N\vee\M^c = \M_0$ if and only if $\C_{\N^c}\cap\C_{\M} = \{\id\}$. Analogous statements hold interchanging $\N$ and $\M$, hence in particular $p\sim q$.
\end{corollary}

\begin{proof}
The category $\C_{\N^c}\cap\C_{\M}$ is automatically
modular with the braiding inherited from $\C$, thanks to Lemma
\ref{lem:modularityintcats}. The first statement follows by the same
argument leading to Proposition \ref{prop:dualoflocalcat} which relies
on the (not necessarily unique) harmonic expansion
(\ref{eq:harmonicabstractexpansion}), on rigidity of
$\C_{\N^c}\cap\C_{\M}$ and on unitarity of its modular $S$-matrix.

Normality of $\N\vee\M^c$ implies $\C_{\N^c}\cap\C_{\M}\subset
\catVec$ as we have seen in Proposition \ref{prop:holotensorsplit},
the converse follows from the normality assumption on charge transporters.

The nontrivial implication in the last statement follows from biduality.
\end{proof}

\section{Abstract points and (Dedekind's) completeness}\label{sec:Dedekind}

In the following we show a way of deriving completeness of the
invariant introduced in Section \ref{sec:braction}, Eq.\
(\ref{eq:resfun}), on the class of \emph{prime} conformal nets. This
section is rather speculative, in the sense that it relies
on two assumptions on the \lq\lq good behaviour" of abstract point (in
the prime CFT case). The first is horizontal and concerns
\emph{transitivity} of the comparability relation $p\sim q$, the
second is vertical and asks \emph{totality} of the unitary equivalence
$p=UqU^*$ encountered in Section \ref{sec:fuzzyapts}. Here we do not
discuss about the issue of deriving them, nor strengthening Definition
\ref{def:apts} or \ref{def:comparapts} in order to do so, 
nor deciding how do they constrain models. We just
show how the structure of the real line (Dedekind's completeness
axiom) and of a conformal net can cooperate in the reconstruction of
the latter up to isomorphism from its abstract points,
thanks to Proposition \ref{prop:totalorderapts}. 

\begin{proposition}\label{prop:ptsandapts}
Let $\{\A\}$ be a prime conformal net on the line (Definition \ref{def:primeconfnet}), fix arbitrarily $I_0\in\cI$ and assume in addition that comparability $p\sim q$ is transitive, and unitary equivalence $p=UqU^*$ is total on the abstract points of $\M_0 = \A(I_0)$. Then $\{\A\}$ is uniquely determined up to isomorphism by its abstract points inside $\M_0$.
\end{proposition}

\begin{proof}
Take first an honest abstract point $p=(\A(I_1),\A(I_2))$ of $\M_0$
with respect to $\{\A\}$, as in Remark \ref{rmk:hpts}. By Remark
\ref{rmk:honestarecomp} all the other honest points are equivalent to
$p$. We want to show that they exhaust the comparability equivalence
class. Let $q=(\N,\N^c)$ be an abstract point of $\M_0$ such that
$q\sim p$, hence by transitivity $q\sim r$ for
every honest point $r = (\A(J_1),\A(J_2))$, and by Proposition \ref{prop:totalorderapts} either $r \leq q$ or $q < r$. Consider the maximum over the first family, i.e., the von Neumann algebra generated by the left relative complements, and the minimum over the second, i.e., the intersection of the left relative complements. The resulting algebras are again honest points because the net is additive and they coincide because the real line is Dedekind complete, thus $q$ is also honest with respect to $\{\A\}$.

Now take an arbitrary abstract point $s=(\M,\M^c)$ of $\M_0$. By the
totality assumption there is a unitary $U\in\cU(\Hilb)$ such that $s =
UpU^*$ where $p=(\A(I_1),\A(I_2))$ as above. Now every unitary is
eligible as an isomorphism of local conformal nets, because positivity
of the energy is preserved by unitary conjugation, hence call
$\{\tilde \A\}$ the net defined on algebras by $\tilde \A(I) :=
U\A(I)U^*$, $I\in\cI$, and observe that
$s=(\tilde\A(I_1),\tilde\A(I_2))$ is an honest point of $\tilde
\A(I_0) = \A(I_0)$ with respect to the new net. As before, $r$
determines all the other honest points (because the comparability
relation and its transitivity property are invariant under
isomorphisms of nets), hence all the local interval algebras 
$\tilde \A(I) \subset \tilde \A(I_0)$, $I\subset I_0$ by taking intersections. 
By Proposition \ref{prop:michiinvariant} the latter determine $\{\tilde \A\}$ 
up to isomorphism, hence $\{\A\}$ as well, and the proof is complete.
\end{proof}

\section{Conclusions}

In chiral conformal QFT, the DHR category $\C=\DHR\{\A\}$
is a unitary braided tensor category corresponding to the positive-energy
representations of the model. In completely rational models, the
braiding is non-degenerate, hence it is a modular tensor category (UMTC). 
While abstract UMTCs are rigid structures and cannot distinguish
the underlying CFT model uniquely, we have studied the question to which
extent the \emph{braided action} of this category on a single (local or global)
algebra $\A$ is a complete invariant of the model. The strategy is to
exploit the trivialization of the braiding, which is a characteristic 
feature of the DHR braiding, in certain geometric
constellations to identify pairs of subalgebras (called ``abstract
points''). They are candidates for subalgebras of local observables
associated to regions (half-intervals or half-lines)
separated by a geometric point. Modularity is needed to distinguish the
left from the right complement, and enters in our analysis through the stronger 
categorical notion of primality for UMTCs. As the main tool in this
direction, we 
established powerful duality relations between subalgebras of $\A$ and
subcategories of $\C$, and a characterization of ``prime'' CFT models
that do not factor through nontrivial subnet, either holomorphic or
not. We formulate a 
unitary equivalence relation and a comparability relation between
abstract points. Assuming that the former is total and the latter is
transitive, we showed that the action of the DHR category is a
complete invariant for prime CFT models, i.e., it allows (in
principle) to reconstruct the local QFT up to unitary equivalence.  

We assumed throughout that the action does come from a
CFT, so that we only have to decide whether two inequivalent CFT can
give rise to the same action. We did not address the more ambitious
question of how to characterize those actions which possibly come from
a CFT, thus leaving the realization problem
of braided actions of abstract UMTCs by DHR categories of some local
net for future research.

\bigskip
\bigskip

{\bf Acknowledgement.} Supported by the German Research Foundation (Deutsche
Forschungsgemeinschaft (DFG)) through the Institutional Strategy of
the University of G\"ottingen. We thank M. Bischoff and R. Longo for drawing
our attention to \cite{Wei11}, which is crucial for Proposition
\ref{prop:ptsandapts}, and to \cite{HaYa00}, which puts our work in a
broader context. We also thank them for their stimulating interest in
this work. We are indebted to Y. Tanimoto for his suggestions, for a
careful proof-reading of an earlier version of this 
manuscript and for pointing out a mistake in our first proof of Proposition
\ref{prop:cd}. We also thank D. Buchholz and R. Conti for motivating conversations.

\small


\end{document}